\documentclass[11pt]{amsart}

\usepackage[latin1]{inputenc}
\usepackage{t1enc}
\usepackage{amsmath, amssymb, amsfonts, amsthm, amsopn,epsfig}
\usepackage[none]{hyphenat}
\usepackage{float}
\usepackage{graphicx}
\usepackage{caption}
\usepackage[toc,page]{appendix}
\usepackage{epstopdf}
\usepackage{etoolbox}
\usepackage{hyperref}
\usepackage{dsfont}
\usepackage[nameinlink, noabbrev, capitalize]{cleveref}
\usepackage[top=25mm, bottom=25mm, left=25mm, right=25mm]{geometry}
\usepackage{comment}
\usepackage{algorithm2e}
\SetKwComment{Comment}{/* }{ */}

\theoremstyle{plain}
\newtheorem{theorem}{Theorem}[section]

\newtheorem{corollary}[theorem]{Corollary}
\newtheorem{claim}[theorem]{Claim}
\newtheorem{lemma}[theorem]{Lemma}

\newtheorem{conjecture}[theorem]{Conjecture}
\newtheorem*{conjecture*}{Conjecture}
\newtheorem*{problem*}{Problem}

\newtheorem*{question*}{Question}

\theoremstyle{definition}
\newtheorem{definition}{Definition}

\DeclareMathOperator{\tr}{tr}
\DeclareMathOperator{\row}{row}
\DeclareMathOperator{\col}{col}
\DeclareMathOperator{\disc}{disc}
\DeclareMathOperator{\herdisc}{herdisc}

\DeclareMathOperator{\D}{D}
\DeclareMathOperator{\R}{R}
\DeclareMathOperator{\Q}{Q}
\DeclareMathOperator{\EQ}{EQ}
\DeclareMathOperator{\surp}{sp}
\DeclareMathOperator{\mc}{mc}

\newcommand{\RomanNumeralCaps}[1]{\MakeUppercase{\romannumeral #1}}
\newcommand{\cF}{\mathcal{F}}

\begin{document}


\title{Factorization norms and an inverse theorem for MaxCut}
    
\author{Igor Balla}
\address{Simons Laufer Mathematical Sciences Institute, 17 Gauss Way, Berkeley, CA 94720.}
\email{iballa1990@gmail.com}
\thanks{The first and third authors acknowledge support by the National Science Foundation under Grant No. DMS-1928930 while the authors were in residence at the Simons Laufer Mathematical Sciences Institute (formerly MSRI) in Berkeley, California, during the Spring 2025 semester.}

\author{Lianna Hambardzumyan}
\address{Department of Computer Science, University of Victoria, Victoria, BC V8P 5C2, Canada.}
\email{liannahambardzumyan@uvic.ca}
\thanks{LH is funded by the Natural Sciences and Engineering Research Council of Canada.}
    
\author{Istv\'an Tomon}
\address{Department of Mathematics and Mathematical Statistics, Ume\r{a} University, Ume\r{a}, 90736, Sweden}
\email{istvan.tomon@umu.se}
\thanks{IT is supported in part by the Swedish Research Council grant VR 2023-03375.}

\sloppy

\begin{abstract}
    We prove that Boolean matrices with bounded $\gamma_2$-norm or bounded normalized trace norm must contain a linear-sized all-ones or all-zeros submatrix, verifying a conjecture of Hambardzumyan, Hatami, and Hatami. We also present further structural results about Boolean matrices of bounded $\gamma_2$-norm and discuss applications in communication complexity, operator theory, spectral graph theory, and extremal combinatorics.

    As a key application, we establish an inverse theorem for MaxCut. A celebrated result of Edwards states that every graph $G$ with $m$ edges has a cut of size at least $\frac{m}{2}+\frac{\sqrt{8m+1}-1}{8}$, with equality achieved by complete graphs with an odd number of vertices. To contrast this, we prove that if the MaxCut of $G$ is at most $\frac{m}{2}+O(\sqrt{m})$, then $G$ must contain a clique of size $\Omega(\sqrt{m})$.
\end{abstract}

\maketitle

\section{Introduction}
For an $m \times n$ matrix $M$, the factorization norm $\gamma_2$ of $M$ is defined as  $$\gamma_2(M)=\min_{U, V : M=UV}\|U\|_{\row}\|V\|_{\col},$$
where $\|U\|_{\row}$ denotes the maximum $\ell_2$-norm of the rows of $U$, and $\|V\|_{\col}$ denotes the maximum $\ell_2$-norm of the columns of $V$. This suggests an intuition that the $\gamma_2$-norm can be viewed as a ``smooth'' version of matrix rank, since the rank is the minimum number of rows of $U$ (which equals to the number of columns of $V$) in a factorization.
Informally, the factorization norm measures how much the action of a matrix $M$ on a vector is distorted  when it is factored through the $\ell_2$ space.

Motivated by connections to communication complexity and combinatorics, this paper investigates the structural properties of $\gamma_2$-norm: given the value of $\gamma_2$-norm for a matrix,  can we structurally characterize the matrix? Conversely, if the matrix has certain structural properties, how large can its $\gamma_2$-norm be? 

Our main result confirms a conjecture of Hambardzumyan, Hatami and Hatami \cite{HHH} which was motivated by an open question in randomized communication complexity (see \cref{sec:cc}). They asked whether a Boolean matrix with bounded $\gamma_2$-norm must contain a linear-sized all-ones or all-zeros submatrix \cite{HHH}. In this work, we answer this question in the affirmative.

\begin{theorem}[Main theorem]\label{thm:mono_rectangle}
    Let $M$ be an $m\times n$ Boolean matrix such that $\gamma_2(M)\leq \gamma$. Then $M$ contains an $\delta_1 m\times \delta_2 n$ all-zeros or all-ones submatrix such that $\delta_1,\delta_2\geq 2^{-O(\gamma^3)}$.
\end{theorem}

More generally, Hambardzumyan, Hatami and Hatami \cite{HHH} proposed an even stronger version of this conjecture where the condition on the matrices is relaxed to having bounded normalized trace norm. Given an $m\times n$ matrix $M$, the \emph{trace norm} of $M$, denoted by $\|M\|_{\tr}$, is the sum of its singular values. Furthermore, the \emph{normalized trace norm} of $M$ is defined as $\frac{\|M\|_{\tr}}{\sqrt{mn}}$ and it is well known that $\gamma_2(M)$ is an upper bound for this quantity. We verify this conjecture by showing that it is just a simple corollary of \cref{thm:mono_rectangle}. 

\begin{corollary}\label{cor:trace}
Let $M$ be an $m\times n$ Boolean matrix such that $\frac{\|M\|_{\tr}}{\sqrt{mn}}\leq \gamma$. Then $M$ contains an $\delta_1 m\times \delta_2 n$ all-zeros or all-ones submatrix such that $\delta_1,\delta_2\geq 2^{-O(\gamma^3)}$.
\end{corollary}

We highlight another immediate corollary of \cref{thm:mono_rectangle} on arbitrary integer matrices of bounded $\gamma_2$-norm.

\begin{corollary}\label{cor:integer}
    Let $M$ be an $m\times n$ integer matrix such that $\gamma_2(M)\leq \gamma$. Then $M$ contains an $\delta_1 m\times \delta_2 n$ constant submatrix such that $\delta_1,\delta_2\geq 2^{-\gamma^{O(\gamma)}}$.
\end{corollary}

Furthermore, we show that \cref{thm:mono_rectangle}  is close to optimal, that is, the factor $2^{-O(\gamma^3)}$ cannot be replaced by anything smaller than $2^{-O(\gamma)}$.


\begin{theorem}\label{prop:construction}
Let $\gamma\geq 3$ and $n$ be sufficiently large with respect to $\gamma$. Then there exists an $n\times n$ Boolean matrix $M$ such that $\gamma_2(M)\leq \gamma$, and $M$ contains no $t\times t$ all-zeros or all-ones submatrix for $t>n2^{-\gamma+3}$.
\end{theorem}

While the $\gamma_2$-norm of a Boolean matrix can, in general, be difficult to estimate, we show that under a certain structural assumption, it is approximately equal to a simple combinatorial parameter. To this end, we define the \emph{degeneracy} of a Boolean matrix $M$ to be the smallest integer $d$ such that every submatrix of $M$ has a row or a column with at most $d$ one entries and we say that a Boolean matrix is \emph{four cycle-free} if it contains no $2\times 2$ all-ones submatrix. In other words, if $M$ is the bi-adjacency matrix of a bipartite graph $G$, then the degeneracy of $M$ is equal to the degeneracy of $G$ and $M$ is four cycle-free if and only if $G$ has no four cycles. 
We show that if $M$ is four cycle-free, then the $\gamma_2$-norm of $M$ is essentially the square-root of its degeneracy. 

\begin{theorem}\label{thm:four_cycle}
Let $M$ be a four cycle-free Boolean matrix of degeneracy $d$. Then $$\gamma_2(M)=\Theta(\sqrt{d}).$$
\end{theorem}

Due to a powerful graph theoretic result of Hunter, Milojevi\'c, Sudakov, and Tomon \cite{HMST0}, \cref{thm:mono_rectangle} and \cref{thm:four_cycle} can be combined to provide the following Zarankiewicz-type result, qualitatively generalizing both.

\begin{theorem}\label{thm:zarankiewicz}
    For every $\gamma>0$ there exists a constant $C >0$ such that the following holds. Let $M$ be an $n\times n$ Boolean matrix such that $\gamma_2(M)\leq \gamma$. If $M$ has at least $Ctn$ one-entries, then $M$ contains a $t\times t$ all-ones submatrix.
\end{theorem}

 The factorization norm $\gamma_2$ and, more broadly, factorization theory for Banach spaces have been a central and influential topic in functional analysis, originating with the work of Kwapie\'n~\cite{kwapien1972operators} and Maurey~\cite{maureytheoremes}, as well as being implicitly present in an earlier work by Grothendieck~\cite{grothendieck1956resume}. Moreover, the $\gamma_2$-norm has also found far-reaching applications across theoretical computer science. These include lower bounds on quantum and randomized communication protocols in communication complexity \cite{LS}, bounds on hereditary discrepancy in discrepancy theory \cite{MNT}, the algorithmic version of the celebrated Bourgain-Tzafriri theorem on the column subset selection problem \cite{tropp2009column}, a connection to margin complexity in learning theory \cite{linial2009learning}, and several applications in differential privacy \cite{muthukrishnan2012optimal, edmonds2020power, fichtenberger2023constant}. Thus, understanding the fundamental properties of the $\gamma_2$-norm is strongly motivated by applications in these areas.

\subsection{The MaxCut problem} Given a graph $G$, a \emph{cut} in $G$ is a partition of the vertices into two parts, together with all the edges having exactly one vertex in each part. The size of the cut is the number of its edges, and the \emph{MaxCut} of $G$ is the maximum size of a cut. Algorithmic and theoretical properties of the MaxCut are extensively studied \cite{AKS,BJS,Edwards,Edwards2,erdos_MaxCut,GoemansWilliamson}.

If $G$ has $m$ edges, then $G$ has a cut of size at least $m/2$, as this is the expected size of the cut resulting from a uniform random partition of the vertex set. A fundamental result of Edwards \cite{Edwards,Edwards2} states that this simple bound can be improved to $m/2+(\sqrt{8m+1}-1)/8$, which is sharp in case $G$ is a complete graph on an odd number of vertices. In general, this shows that every graph with $m$ edges has MaxCut of size at least $m/2+\Omega(\sqrt{m})$. On the other hand, all known examples of graphs with MaxCut of size $m/2+O(\sqrt{m})$ are close to complete graphs or the disjoint union of complete graphs. Motivated by this, we establish a strong structural property of graphs with MaxCut of size $m/2+O(\sqrt{m})$: they contain a complete subgraph of size $\Omega(\sqrt{m})$.

\begin{theorem}\label{thm:main_MaxCut}
    Let $\alpha>0$ and let $G$ be a graph with $m$ edges containing no cut of size larger than $m/2+\alpha\sqrt{m}$. Then $G$ contains a clique of size $2^{-O(\alpha^9)}\sqrt{m}$.
\end{theorem}


This result connects to a line of research initiated by Erd\H{o}s and Lov\'asz in the 1970s (see \cite{erdos_MaxCut}), who studied MaxCut under forbidden subgraph conditions. Specifically, they considered the maximum cut size in graphs that avoid a fixed subgraph $H$, a problem which received substantial interest \cite{Alon_MaxCut,AKS,BJS,GJS}. Alon, Krivelevich, and Sudakov \cite{AKS} showed that for every graph $H$ there exists a constant $\varepsilon_H>0$ such that every $H$-free graph with $m$ edges has a cut of size at least $m/2+\Omega_H(m^{1/2+\varepsilon_H})$. They further conjectured that a stronger bound $m/2+\Omega_H(m^{3/4+\varepsilon_H'})$ holds for some appropriate $\varepsilon_H'>0$. In order to prove this, it is clearly enough to consider the case when $H$ is a complete graph. A closely related result of R\"aty, Sudakov, and Tomon \cite{RST} shows that regular graphs with edge density between $1/2+\varepsilon$ and $1-\varepsilon$ have a cut of size $m/2+\Omega_{\varepsilon}(m^{5/8})$.

Our \cref{thm:main_MaxCut} addresses the extreme case of the Alon, Krivelevich, Sudakov conjecture, where $H$ is a clique whose size is comparable to the host graph, specifically $H=K_{\Omega(\sqrt{m})}$.

At first glance, \cref{thm:main_MaxCut} may appear unrelated to factorization norms.  However, somewhat surprisingly, our proof of this theorem relies on \cref{thm:mono_rectangle} and \cref{cor:trace}. Specifically, we show that if $A$ is the adjacency matrix of $G$, then $G$ has a cut of size $m/2+\Omega(\|A\|_{\tr})$, improving a recent result of \cite{RatyTomon}. From this, we get an upper bound on the trace norm of $A$, enabling a direct application of \cref{cor:trace}. The proof of \cref{thm:main_MaxCut} is presented in \cref{sect:maxcut}.

\medskip
 \textbf{Organization.} In \cref{sect:applications}, we present applications and motivations of our results in communication complexity, operator theory, combinatorics and discrepancy theory. We prove \cref{thm:mono_rectangle}, \cref{cor:trace} and \cref{cor:integer} in \cref{sect:mono_rectangle} after some preliminary results in \cref{sect:sparsifying}. Then, we prove \cref{prop:construction} in \cref{sect:construction}, \cref{thm:four_cycle} and \cref{thm:zarankiewicz} in \cref{sect:four_cycle}, and \cref{thm:main_MaxCut} in \cref{sect:maxcut}.

\section{Applications}\label{sect:applications}

\subsection{Communication complexity}\label{sec:cc}
The $\gamma_2$-norm was introduced into communication complexity by the seminal paper of Linial and Shraibman \cite{LS}. The $\gamma_2$-norm of the matrix $M$ and its approximate version lower bound the following basic and well-studied communication models: \\ deterministic $\D(M)$, deterministic with oracle access to the Equality function $\D^{\EQ}(M)$, public-coin  randomized communication $\R(M)$, and quantum communication with shared entanglement $\Q^{*}(M)$. More precisely, for a Boolean matrix $M$, we have 
\[\log \gamma_2(M) \leq \D^{\EQ}(M) \leq \D(M),\]
where the first inequality is proven in \cite{HHH} and the weaker inequality of $\log \gamma_2(M) \leq \D(M)$ was initially proven in~\cite{LS}.

For the $n \times n$ matrix $M$, let the \emph{approximate} $\gamma_2$\emph{-norm} of $M$, denoted by $\tilde{\gamma}_2(M)$, be the minimum $\gamma_2$-norm of an $n \times n$ matrix $M'$ that satisfies $|M(i,j)-M'(i,j)|\leq \frac{1}{3}$ for every entry $(i,j)\in [n]\times [n]$. Then, the following inequality is proved in \cite{LS},
\begin{equation}\label{eq:gamma-random}
    \log \tilde{\gamma}_2(M) \lesssim \Q^{*}(M)\leq \R(M) \leq O(\tilde{\gamma}_2(M))^2.
\end{equation}

\subsubsection{Structure and randomized communication.} All-ones and all-zeros submatrices of a Boolean matrix play a central role in communication complexity, since they serve as the fundamental building blocks of communication protocols. Finding such submatrices of large size often leads to efficient protocol design via recursion techniques (e.g., \cite{nisan1995rank}), while their absence can lead to strong lower bounds against communication protocols in various models.  For instance, a classical result states that if the deterministic communication complexity of a Boolean matrix is low, then the matrix can be partitioned into large all-ones and all-zeros submatrices. This structural insight has been instrumental in obtaining strong lower bounds in the deterministic model.

In contrast, no such decomposition holds in general for randomized communication complexity -- this is evident via the simple example of the identity matrix. Nevertheless, one may still hope to find large structured submatrices in a Boolean matrix with bounded randomized complexity. This intuition is formalized in the following conjecture, stated in its most basic setting:
\begin{conjecture}[\cite{CLV,HHH}]\label{conj:randomized}
Every Boolean matrix with randomized communication
complexity bounded by a constant contains a linear-sized all-zeros or all-ones submatrix.
\end{conjecture}
In view of the lower bound in~(\ref{eq:gamma-random}), one natural strategy to settle the conjecture would be to show that if the matrix has bounded approximate $\gamma_2$-norm, then it has a linear-sized all-zeros or all-ones submatrix. However, until now, this question has remained open even for the exact $\gamma_2$-norm itself. As such, this problem has stood as a major barrier to progress on  \cref{conj:randomized}. Motivated by this obstacle, Hambardzumyan, Hatami, and Hatami \cite[Conjecture \RomanNumeralCaps{2}]{HHH} conjectured that such structure exists not only for matrices with bounded $\gamma_2$-norm but even for the ones that satisfy the relaxed condition of having a bounded normalized trace norm. This is indeed a relaxation, as 
$\frac{\|M\|_{\tr}}{\sqrt{mn}} \leq \gamma_2(M)$. 
In \cite{HHH}, this conjecture regarding the normalized trace norm was proven for a special class of matrices known as group lifts, for which it can be deduced from Cohen's idempotent theorem. 

In this work, we resolve the conjecture of \cite{HHH} regarding matrices with bounded normalized trace norm in full generality. Our proof techniques are purely combinatorial and linear algebraic. Specifically,~\cref{cor:trace} shows that every Boolean matrix with bounded normalized trace norm contains a linear-sized all-zeros or all-ones submatrix.
This result eliminates a major bottleneck toward  \cref{conj:randomized} and offers strong evidence in its favor. 


\subsubsection{Separation between the $\gamma_2$-norm and randomized communication.} Linial and Shraibman \cite{LS} proposed the problem of whether $\tilde{\gamma}_2(M)$ in~\eqref{eq:gamma-random} can be replaced with $\gamma_2(M)$ to get a stronger lower bound for $\R(M)$. However, this was recently disproved by  Cheung,  Hatami, Hosseini, and Shirley \cite{CHHS} in a strong sense, who constructed an $n\times n$  Boolean matrix $M$ such that $\gamma_2(M)\geq \Omega(n^{1/32})$ and $\R(M)=O(\log \log n)$. Their main technical result is as follows. 

Let $1\leq q\leq p$ be integers, and let $P=P(q,p)$ be the $qp\times qp$ Boolean matrix, whose rows and columns are indexed by the elements of $[q]\times \{0,\dots,p-1\}$, and its entries are given by $P[(x,x'),(y,y')]=1$ iff $xy+x'=y'$. Furthermore, let $P_p=P_p(q,p)$ be the matrix defined almost identically, but $P[(x,x'),(y,y')]=1$ iff $xy+x'=y'$ holds modulo $p$. In \cite{CHHS}, it is proved, by technical applications of Fourier analysis, that $\gamma_2(P_p)=\Omega(q^{1/8})$ if $q\leq \sqrt{p}$, and $\gamma_2(P)=\Omega(q^{1/8})$ if $q\leq p^{1/3}$.

However, note that $P$ and $P_p$ are the incidence matrices of points and lines, so they are four cycle-free. Therefore, \cref{thm:four_cycle} immediately implies the following improvements.

\begin{theorem}
Let $1\leq q\leq p-1$. Then  $\gamma_2(P_p)=\Theta(\sqrt{q})$ and $\gamma_2(P)=\Theta(\min\{\sqrt{q},p^{1/4}\})$.
\end{theorem}

\begin{proof}
 Given $x,x',y$, there is a unique $y'$ such that $xy+x'=y' \pmod{p}$, and also given $x,y,y'$, there is a a unique $x'$ such that $xy+x'=y' \pmod{p}$. Therefore, each row and column of $P_p$ contains $q$ one entries, so the degeneracy of $P_p$ is also $q$. By \cref{thm:four_cycle}, we get $\gamma_2(P_p)=\Theta(\sqrt{q})$.

Now let us consider $P$, and let us only prove the lower bound, we leave the upper bound as an exercise. We may assume that $q\leq \sqrt{p}$, as otherwise $P(\sqrt{p},p)$ is a submatrix of $P(q,p)$ and we use that the $\gamma_2$-norm of a submatrix is always at most the $\gamma_2$-norm of the matrix. Given $x\in [q]$, there are at least $qp/4$ solutions of $xy+x'=y'$ with $x,y\in [q]$ and $x',y'\in \{0,\dots,p-1\}$. Therefore, the number of one entries of $P$ is at least $q^2/4$, which means that the degeneracy of $P$ is at least $q/4$. Hence, by \cref{thm:four_cycle}, $\gamma_2(P_p)=\Omega(\sqrt{q})$. 
\end{proof}

Very recently, a similar result was obtained by Cheung, Hatami, Hosseini, Nikolov, Pitassi, and Shirley \cite{CHHNPS}, based on similar ideas. One of their main technical lemmas shows that if $M$ is a four-cycle free Boolean matrix, then $\gamma_2(M)\geq ||M||_F^2/\sqrt{2\Delta}$, where $\Delta$ is the maximum degree of the associated bipartite graph. This gives the same bound as \cref{thm:four_cycle} in case the bipartite graph is close to regular and otherwise, \cref{thm:four_cycle} is stronger. 

\subsubsection{The Log-rank conjecture} The celebrated Log-rank conjecture of Lov\'asz and Saks \cite{logrank} proposes that the deterministic communication complexity of a rank $r$ matrix is bounded by $(\log r)^{O(1)}$. This conjecture is equivalent to the statement that every $m\times n$ Boolean matrix $M$ of rank $r$ contains an all-ones or all-zeros submatrix of size $m2^{-(\log r)^{O(1)}}\times n2^{-(\log r)^{O(1)}}$. This conjecture is still wide open, where the best known lower bound on the size of the all-ones or all-zeros submatrix is $m2^{-O(\sqrt{r})}\times n2^{-O(\sqrt{r})}$, due to a recent result of Sudakov and Tomon \cite{ST24}, slightly improving an earlier result of Lovett \cite{L16}. 

The Log-rank conjecture is closely related to \cref{thm:mono_rectangle}. Indeed, as observed by Linial and Shraibman \cite{LS}, if $M$ is a Boolean matrix of rank $r$, then $\gamma_2(M)\leq \sqrt{r}$. Therefore, our main result in   \cref{thm:mono_rectangle} yields a submatrix comparable to the best known bounds on the Log-rank conjecture, but it is also applicable to a much larger class of matrices. On the other hand, \cref{prop:construction} shows that one cannot hope to settle the Log-rank conjecture -- or even improve the current best bound -- by obtaining a larger than $m2^{-O(\sqrt{r})}\times n2^{-O(\sqrt{r})}$ all-ones or all-zeros submatrix solely by tightening the dependence on $\gamma_2$ in \cref{thm:mono_rectangle}.

 \subsection{Operator theory and harmonic analysis} In operator theory, the question of characterizing the idempotent Schur multiplies in terms of contractive idempotents is widely open. In \cite{HHH}, using the fact that the operator norm induced by a Schur multiplier 
$M$ coincides with its $\gamma_2$-norm, it was shown that this question is equivalent to the structural characterization of matrices with bounded $\gamma_2$-norm. In particular, it is equivalent to the following conjecture. A \emph{blocky-matrix} is a blow-up of a permutation matrix.
 \begin{conjecture}[\cite{HHH}]\label{conj:blocky}
     Let $M$ be an $m\times n$ Boolean matrix such that $\gamma_2(M)\leq \gamma$. Then there exists  $c_{\gamma}$, depending only on $\gamma$, such that $M$ is a $\pm 1$-linear combination of at most $c_{\gamma}$ blocky-matrices.
 \end{conjecture}
  It is not difficult to see that~\cref{conj:blocky}, if true, implies~\cref{thm:mono_rectangle} (see \cite[Lemma 3.5]{HHH}). On the other hand, in an upcoming work,  Goh and Hatami \cite{GH2} builds on ~\cref{thm:mono_rectangle} to prove the following substantial evidence towards this conjecture. They show that if $M$ is a Boolean matrix with $|M|$ one-entries such that $\gamma_2(M)\leq \gamma$, then there is a blocky-matrix $B$ such that  $|B|\geq |M|/2^{2^{O(\gamma)}}$, and $B$ is ``contained'' in $M$, i.e. $M_{i,j}\geq B_{i,j}$ for every entry $(i,j)$. In another, even more recent work, Goh and Hatami \cite{GH3} prove that \cref{conj:blocky} holds up to a polylogarithmic error term. More precisely, $M$ can be written as a signed sum of at most $2^{O(\gamma^7)}(\log\min\{m,n\})^2$ blocky-matrices.

 From the perspective of harmonic analysis, ~\cref{conj:blocky} is the analogue of Cohen's celebrated idempotent theorem for the algebra of Schur multipliers. Specifically, Cohen's theorem -- made quantitative for finite Abelian groups by Green-Sanders \cite{green2008boolean, green2008quantitative} -- states that if $f : G \to \{0,1\}$ is a function on a finite Abelian group $G$ with bounded Fourier $\ell_1$-norm, $\|\hat{f}\|_1 \leq \ell$, then $f$ can be written as the $\pm 1$-linear combination of $c_{\ell}$ many indicator functions of cosets of $G$, that is $f = \sum_{i}^{c_{\ell}} \pm 1_{H_{i} + a_i}$. Equivalently, every idempotent of the Fourier algebra of $G$ can be
expressed as a linear combination of  many contractive idempotents $c_{\ell}$ -- precisely the type of structural result one seeks in the algebra of Schur multipliers. For more details on this connection, see \cite{HHH}.

  Another connection to Cohen's idempotent theorem is for a special class of matrices known as group lifts, that is, matrices $M$ that can be written as $M(x,y) = f(y - x)$ for some function $f : G \to \mathbb{R}$ over a finite group $G$. Under such lift, we know that $\|\hat{f}\|_1 = \gamma_2(M)$, and an indicator function of a coset corresponds to a blow-up of an identity matrix. Hence,~\cref{conj:blocky} holds for group lifts as a direct consequence of Cohen's theorem.

\subsection{Spectral graph theory} Graphs, whose adjacency matrices have bounded $\gamma_2$-norm naturally appear in spectral graph theory, and in the study of equiangular lines.

\subsubsection{Graphs of bounded smallest eigenvalue.} A central topic of spectral graph theory is to understand the structure of graphs having smallest eigenvalue at least $-\lambda$, see e.g.\ Koolen, Cao and Yang \cite{KCQ} for a survey. It is easy to show that if a $G$ is a nonempty graph, then the smallest eigenvalue of its adjacency matrix is at most $-1$, with equality if and only if $G$ is the disjoint union of cliques. A celebrated theorem of Cameron, Goethels, Seidel, and Shult \cite{CGSS} from 1972 settles the case $\lambda = 2$ and establishes a connection to root systems, and more recently, Koolen, Yang and Yang \cite{KYY} obtained a partial characterization in the case $\lambda = 3$. 

It turns out that if a graph has smallest eigenvalue at least $-\lambda$, then the $\gamma_2$-norm of its adjacency matrix is at most $2\lambda$, see  \cref{lemma:smallest_eigenvalue}. Hence, our results immediately apply to such graphs. Moreover, we establish the following strengthening of  \cref{thm:zarankiewicz}, which shows that graphs of bounded smallest eigenvalue contain cliques of size comparable to their average degree. This result also follows from a structure theorem of Kim, Koolen and Yang \cite{KKY}. However, while their proof uses Ramsey theoretic arguments, our proof relies on \cref{thm:zarankiewicz} instead.

\begin{theorem}\label{thm:smallest_eigenval}
Let $\lambda>0$ and let $G$ be a graph with average degree $d$ and smallest eigenvalue at least $-\lambda$. Then $G$ contains a clique of size $\Omega_{\lambda}(d)$.
\end{theorem}

We present the proof of \cref{thm:smallest_eigenval} in \cref{sect:smallest_eig}.

\subsubsection{Equiangular lines} A collection of lines $\mathcal{L}$ in $\mathbb{R}^d$ through the origin is \emph{equiangular} with angle $\alpha$ if the angle between any two lines of $\mathcal{L}$ is $\alpha$. It is an old problem of Lemmens and Siedel \cite{LemmensSiedel} to determine the maximum  size of an equiangular set of lines with angle $\alpha$ in $\mathbb{R}^d$. This problem was essentially solved recently by Jiang, Tidor, Yao, Zhang and Zhao \cite{JTYZZ}, building on the work of  Balla,  Dr\"axler, Keevash, and Sudakov \cite{BDKS}. For the extensive history of equiangular lines and further quantitative improvements on this problem, we refer the interested reader to \cite{Balla}.

A collection of equiangular lines with angle $\alpha$ can be represented by a symmetric Boolean matrix (or equivalently a graph) as follows. Pick a unit direction vector $v_{\ell}$ for each line $\ell$, and let $M$ be the matrix whose rows and columns are represented by the elements of $\mathcal{L}$, and set $M(\ell_1,\ell_2)=1$ if the angle between $v_{\ell_1}$ and $v_{\ell_2}$ is $\alpha$, and set $M(\ell_1,\ell_2)=0$ if this angle is $\pi-\alpha$. Set the diagonal entries 0. If $t=\cos(\alpha)$, then $M=\frac{1}{2t}N-\frac{1}{2t}I+J$, where $N$ is the Gram matrix of the vectors $\{v_{\ell}\}_{\ell\in \mathcal{L}}$. From this, $\gamma_2(N)\leq 1/t+1$. Hence, if $\alpha$ is bounded away from $\pi/2$, the Boolean matrix $M$ has bounded $\gamma_2$-norm, so our results immediately apply, giving interesting results about the Ramsey properties of the graphs associated with equiangular lines. In particular, we get the following immediate corollary of \cref{thm:mono_rectangle}.

\begin{theorem}
    Let $\delta>0$, then there exists $c=c(\delta)>0$ such that the following holds. Let $\alpha\in [0,\pi/2-\delta)$, and let $\mathcal{A}$ and $\mathcal{B}$ be two sets of lines through the origin in a real space such that the angle between $\ell_1$ and $\ell_2$ is $\alpha$ for every $(\ell_1,\ell_2)\in \mathcal{A}\times \mathcal{B}$. Then, given any choice of direction vectors $v_{\ell}$ for every $\ell\in\mathcal{A}\cup \mathcal{B}$, there exist $\mathcal{A}'\subset \mathcal{A}$ and $\mathcal{B}'\subset\mathcal{B}$ such that $|\mathcal{A}'|\geq c|\mathcal{A}|$, $|\mathcal{B}'|\geq c|\mathcal{B}|$, and either,  the angle between $v_{\ell_1}$ and $v_{\ell_2}$ is $\alpha$ for every $(\ell_1,\ell_2)\in\mathcal{A}'\times \mathcal{B}'$, or it is $\pi-\alpha$ for every $(\ell_1,\ell_2)\in\mathcal{A}'\times \mathcal{B}'$.
\end{theorem}

 We highlight that the study of  Ramsey properties of graphs associated to equiangular lines played a crucial role in the resolution of the above described problem of Lemmens and Siedel, see \cite{BDKS,JTYZZ} for further details.

\subsection{Discrepancy theory}
The $\gamma_2$-norm has important applications in discrepancy theory as well. Let $M$ be an $m\times n$ matrix, then the \emph{discrepancy} (also referred to as combinatorial discrepancy) of $M$ is defined as
$$\disc(M)=\min_{x\in \{-1,1\}^n} ||Mx||_{\infty}.$$
Here, $\|\cdot\|_{\infty}$ is the maximum absolute value of the entries. Moreover, the \emph{hereditary discrepancy} of $M$ is defined as $\herdisc(M)=\max_{N\subset M}  \disc(N),$ where the maximum is taken over all submatrices $N$ of $M$.  If $\cF$ is set system on a ground set $X$, then $\disc(\cF)=\disc(M)$ and $\herdisc(\cF)=\herdisc(M)$, where $M$ is the incidence matrix of $\cF$, with rows representing the sets. In combinatorial terms, the discrepancy of $\cF$ is the minimal $k$ for which there is a red-blue coloring of the elements of $X$ such that the numbers of red and blue elements in each set of $\cF$ differ by at most $k$.

Combinatorial discrepancy theory has its roots in the study of irregularities of distributions and it has become a highly active area of research since the 80's \cite{BC}. It has also found profound applications in computer science, see the book of Chazelle \cite{Ch} for a general reference. The following general inequality of Matou\v{s}ek, Nikolov, and Talwar \cite{MNT} establishes a sharp relation between the $\gamma_2$-norm and the hereditary discrepancy of arbitrary matrices:
    $$\Omega\left(\frac{\gamma_2(M)}{\log m}\right)\leq \herdisc(M)=O(\gamma_2(M)\sqrt{\log m}).$$
Combining this theorem with \cref{thm:four_cycle}, we immediately get that if $M$ is a four cycle-free Boolean matrix of degeneracy $d$, then $\herdisc(M)$ and $\sqrt{d}$ are equal up to logarithmic factors. For example, if $G$ is the incidence graph of $n$ points and $m$ lines in the plane, then $G$ is four cycle-free and the Szemer\'edi-Trotter theorem implies that it has degeneracy $O(n^{1/3})$. This bound is also the best possible, so we get close to optimal bounds on the discrepancy of geometric set systems generated by lines, recovering the results of Chazelle and Lvov \cite{ChL}.

\section{Preliminaries}
In this paper, we use mostly standard linear algebraic and graph theoretic notation. 

\subsection{Graph theory} Given a graph $G$, $e(G)=|E(G)|$ denotes the number of edges of $G$. If $U\subset V(G)$, then $G[U]$ is the subgraph of $G$ induced on the vertex set $U$. We make use of the following fundamental result of extremal graph theory.

\begin{theorem}[Tur\'an's theorem \cite{turan}]\label{thm:turan}
    Let $G$ be a graph on $n$ vertices such that the complement of $G$ has average degree at most $d$. Then $G$ contains a clique of size at least
    $\frac{n}{d+1}.$
    Equivalently, if the complement has $t$ edges, then there is a clique of size  at least
    $\frac{n^2}{2t+n}.$
\end{theorem}

\subsection{Basics of linear algebra} Given a real vector $v$, $||v||$ denotes the $\ell_2$-norm of $v$. The \emph{Hadamard product} (or entry-wise product) of two matrices $A$ and $B$ of size $m\times n$ is the $m\times n$ sized matrix $A\circ B$ defined as $(A\circ B)_{i,j}=A_{i,j}B_{i,j}$. We make use of the well known \emph{Cauchy interlacing theorem}, which we state here for the reader's convenience.

\begin{lemma}\label{lemma:cauchy}
Let $M$ be an $n\times n$ real symmetric matrix with eigenvalues $\lambda_1\geq \dots\geq \lambda_n$, and let $N$ be an $m\times m$ principal submatrix of $M$ with eigenvalues $\mu_1\geq \dots\geq \mu_m$. Then $\lambda_i\geq \mu_i\geq \lambda_{i+n-m}$ for $i=1,\dots,m$.
\end{lemma}

\subsection{Matrix norms}
Given two matrices $A,B\in \mathbb{R}^{m\times n}$, their Frobenius inner product is defined as $$\langle A,B\rangle=\sum_{i=1}^m\sum_{j=1}^n A_{i,j}B_{i,j}=\tr(AB^T).$$ Let $M$ be an $m\times n$ matrix with singular values $\sigma_1,\dots,\sigma_k$, $k=\min\{m,n\}$. The \emph{Frobenius norm} of $M$ can be defined in multiple equivalent ways: $$||M||_F^2:=\langle M,M\rangle=\tr(MM^T)=\sum_{i=1}^{k}\sigma_i^2.$$ The \emph{trace-norm} of $M$ is $$||M||_{\tr}=\sigma_1+\dots+\sigma_k.$$ Finally, we write $$|M|=\sum_{i,j}|M_{i,j}|,$$ which, in case $M$ is a Boolean matrix, is just the number of one entries.

\subsection{The $\gamma_2$-norm}\label{sec:gamma2-properties}
The $\gamma_2$-norm of a matrix $M$ is defined as 
$$\gamma_2(M)=\min_{M=UV}||U||_{\row}||V||_{\col},$$
where $||U||_{\row}=||U||_{2\rightarrow\infty}$ denotes the maximum $\ell_2$-norm of the rows of $U$, and $||V||_{\col}=||V||_{1\rightarrow 2}$ denotes the maximum $\ell_2$-norm of the columns of $V$. 
Here, we collect some basic properties of the $\gamma_2$-norm, see e.g. \cite{LSS} as a general reference.

Let $M\in \mathbb{R}^{m\times n}$ and let $N$ be a real matrix.
\begin{enumerate}
    \item If $c\in \mathbb{R}$, then $\gamma_2(cM)=|c|\gamma_2(M)$.
    \item (monotonicity) If $N$ is a submatrix of $M$, then $\gamma_2(N)\leq \gamma_2(M)$.
    \item (subadditivity) If $M$ and $N$ have the same size, then $\gamma_2(M+N)\leq \gamma_2(M)+\gamma_2(N)$.
    \item\label{pr:gamma}  $\gamma_2(M)=\max||M\circ (u v^T)||_{\tr},$ where the maximum is over all unit vectors $u\in \mathbb{R}^m$, $v\in \mathbb{R}^{n}$.
    \item  $\gamma_2(M)\geq \frac{1}{\sqrt{mn}}||M||_{\tr}.$
    \item Duplicating rows or columns of $M$ does not change the $\gamma_2$-norm.
    \item $\gamma_2(M)\leq ||M||_{\row}$ and $\gamma_2(M)\leq ||M||_{\col}$.
    \item If $M$ is a non-zero Boolean matrix, then $\gamma_2(M)\geq 1$, with equality if and only if $M$ is the blow-up of a permutation matrix.
    \item If $M$ is a Boolean matrix, then $\gamma_2(M)\leq \sqrt{\mbox{rank}(M)}$.
    \item\label{pr:tensor} $\gamma_2(M\otimes N)=\gamma_2(M)\gamma_2(N)$, where $\otimes$ denotes the tensor product.
    \item\label{pr:hadamard} If $M$ and $N$ have the same size, then $\gamma_2(M\circ N)\leq \gamma_2(M)\gamma_2(N)$
\end{enumerate}

Here, we remark that (\ref{pr:hadamard}) follows from (\ref{pr:tensor}) as $M\circ N$ is a submatrix of $M\otimes N$.

\section{Sparsifying matrices}\label{sect:sparsifying}

In this section, we prove a technical result showing that if $M$ is a Boolean matrix with bounded $\gamma_2$-norm such that the density of zero entries is $\varepsilon>0$, then one can boost this to density at least $1-\varepsilon$ by passing to a linear-sized submatrix of $M$. Such a result can be proved in at least two different ways. One approach is based on noting that Boolean matrices of bounded $\gamma_2$-norm have bounded VC-dimension, and then invoke the \emph{ultra-strong regularity lemma} for such matrices, see e.g. \cite{FPS19}. Another approach is based on the discrepancy method, recently used in connection to the Log-rank conjecture \cite{HMST,L16,ST24}. We present the latter approach, as it provides quantitatively better bounds. We mostly follow the proof of Lemma 4.2 in \cite{HMST}.

Given an $m\times n$ Boolean matrix $M$, let $p(M)=|M|/mn$, that is, $p(M)$ is the density of one entries. Define the \emph{discrepancy} of $M$ as
$$\disc(M)=\max_{A\subset [m], B\subset [n]} \big||M[A\times B]|-p(M)|A||B|\big|.$$

\begin{lemma}\label{lemma:disc}
Let $M$ be a Boolean matrix such that $p(M)<1-\varepsilon$ and $\gamma_2(M)\leq \gamma$. Then
$$\disc(M)=\Omega(\varepsilon^2 |M|/\gamma).$$
\end{lemma}

\begin{proof}
Let $m\times n$ be the size of $M$, and let $N=M-p(M)J$. Then $$\disc(M)=\max_{A\subset [m], B\subset [n]} |N[A\times B]|,$$ where the right-hand-side is just the  \emph{cut-norm} of $N$. The dual of the $\gamma_2$-norm is defined as $$\gamma_2^*(N)=\max \Big|\sum_{i=1}^{m}\sum_{j=1}^n N_{i,j}\langle x_i,y_j\rangle\Big|,$$
where the maximum is taken over all vectors $x_1, \dots, x_m, y_1, \dots, y_n\in \mathbb{R}^{d}$ and all dimensions $d$ such that $\|x_i\|, \|y_j\|\leq 1$. It follows from Grothendieck's inequality that the $\gamma_2^*$-norm and cut-norm are equal up to absolute constant factors (see e.g. \cite{LS}), so we have
$$\disc(M)=\Omega(\gamma_2^*(N)).$$
We have $\gamma_2(N)\leq \gamma_2(p(M)J)+\gamma_2(M)\leq 1+\gamma$, so there exists a factorization $N=UV$ such that $||U||_{\row}\leq 1+\gamma$ and $||V||_{\col}\leq 1$. For $i=1,\dots,m$, let $x_i=u_i/(1+\gamma)$, where $u_1,\dots,u_m$ are the rows of $U$, and let $y_1,\dots,y_n$ be the columns of $V$. Then $||x_i||,||y_j||\leq 1$, so
$$\gamma_2^{*}(N)\geq \Big|\sum_{i=1}^{m}\sum_{j=1}^n N_{i,j}\langle x_i,y_j\rangle\Big|=\frac{1}{\gamma+1}\sum_{i=1}^m\sum_{j=1}^n N_{i,j}^2=\frac{1}{\gamma+1}\sum_{i=1}^m\sum_{j=1}^n (M_{i,j}-p(M))^2\geq \frac{\varepsilon^2|M|}{\gamma+1}.$$
This finishes the proof.
\end{proof}

We also use the following consequence of Claim 2.2 from \cite{ST24}.

\begin{lemma}\label{lemma:half_sized}
Let $M$ be an $m\times n$ binary matrix. Then $M$ contains an $m/2\times n/2$ submatrix $N$ such that $p(N)\leq p(M)-\Omega(\frac{\disc(M)}{mn})$.
\end{lemma}

Combining the previous two lemmas, we deduce the following.
\begin{lemma} \label{lemma:disc_decrement}
Let $M$ be an $m\times n$ Boolean matrix such that $\gamma_2(M)\leq \gamma$, and $p(M)\leq 1-\varepsilon$. Then $M$ contains an $m/2\times n/2$ submatrix $N$ such that $p(N)\leq p(M)(1-\Omega(\varepsilon^2/\gamma))$.
\end{lemma}

\begin{proof}
By \cref{lemma:half_sized},  $M$ contains an $m/2\times n/2$ submatrix $N$ such that $p(N)\leq p(M)-\Omega(\frac{\disc(M)}{mn})$. However, by \cref{lemma:disc}, we have $\disc(M)=\Omega(\varepsilon^2|M|/\gamma)=\Omega(\varepsilon^2 p(M)mn/\gamma)$, finishing the proof.
\end{proof}

\begin{lemma}\label{lemma:sparsifying}
Let $M$ be an $m\times n$ Boolean matrix such that $\gamma_2(M)\leq \gamma$, and $p(M)\leq 1-\varepsilon$, and let $0<\delta<1/2$. Then $M$ contains an $m'\times n'$ submatrix $N$ such that every row of $N$ has at most $\delta n'$ one entries, every column has at most $\delta m'$ one entries, and $m'/m, n'/n > 2^{-O((\gamma/\varepsilon^2)\log(1/\delta))}$.
\end{lemma}

\begin{proof}
Define a sequence of submatrices $M_0,M_1,\dots$ of $M$ as follows. Let $M_0=M$. If $M_i$ of size $m_i\times n_i$ is already defined such that $p_i=p(M_i)$, let $M_{i+1}$ be the $m_{i}/2\times n_{i}/2$ sized submatrix of $M_i$ given by ~\cref{lemma:disc_decrement}, so that $p(M_{i+1})\leq p(M_i)(1-\Omega(\varepsilon^2/\gamma))$. Then $m_i=m2^{-i}$, $n_i=n2^{-i}$, and 
$$p(M_i)\leq p(M)(1-\Omega(\varepsilon^2/\gamma))^i\leq p(M)e^{-\Omega(i\varepsilon^2/\gamma)}.$$ Hence, if $I = \left\lceil c\frac{\gamma}{\varepsilon^2} \log (1/\delta) \right\rceil$ for some $c$ sufficiently large, we have $p(M_I)\leq \delta/4$. Delete all rows of $M_I$ which contain more that $\delta n_I/2$ one entries, and delete all columns of $M_I$ which contain more than $\delta m_I/2$ one entries. Then we deleted at most $m_I/2$ rows and $n_I/2$ columns. If we let $N$ be the resulting matrix, then $N$ has size $m'\times n'$ with $m'\geq m_I/2= 2^{-I-1}m$ and $n'\geq n_I/2=2^{-I-1}n$, and each row of $N$ contains at most $\delta m_I/2 \leq \delta m'$ one entries, and each column contains at most $\delta n_I/2 \leq \delta n'$ one entries. 
\end{proof}

\section{Large all-ones or all-zeros submatrices}\label{sect:mono_rectangle}

In this section, we prove \cref{thm:mono_rectangle} and \cref{cor:trace}. First, we show that  \cref{thm:mono_rectangle} indeed implies \cref{cor:trace}. To this end, we prove the following simple lemma.

\begin{lemma}\label{lemma:trace_to_gamma}
Let $M$ be an $m\times n$ matrix such that $||M||_{\tr}\leq \gamma \sqrt{mn}$. Then for every $\varepsilon\in (0,1/2]$, $M$ contains an $m'\times n'$ submatrix $M'$ such that $m'\geq (1-\varepsilon)m$, $n'\geq (1-\varepsilon)n$, and $\gamma_2(M')\leq \gamma/\varepsilon$.
\end{lemma}

\begin{proof}
Let $k=\min\{m,n\}$, and let $\sigma_1,\dots,\sigma_k$ be the singular values of $M$. Let $M=A\Sigma B$ be the singular value decomposition of $M$, that is, $A$ is an $m\times k$ matrix, whose columns are orthonormal vectors, $B$ is $k\times n$ matrix, whose rows are orthonormal vectors, and $\Sigma$ is a $k\times k$ diagonal matrix, whose diagonal entries are $\sigma_1,\dots,\sigma_{k}$. Let $U=A\Sigma^{1/2}$ and $V=\Sigma^{1/2} B$, so $M=UV$. Observe that $$||U||_F^2=||V||_F^2=\sum_{i=1}^k\sigma_i=||M||_{\tr}\leq \gamma \sqrt{mn}.$$ Let $U'$ be the submatrix of $U$ we get by keeping the $m'=(1-\varepsilon)m$ rows with smallest $\ell_2$-norms, and let $V'$ be the submatrix of $V$ we get by keeping the $n'=(1-\varepsilon) n$ columns with the smallest $\ell_2$-norms. Then $||U'||_{\row}^2\leq (\gamma/\varepsilon)\sqrt{n/m}$ and $||V'||_{\col}^2\leq (\gamma/\varepsilon)\sqrt{m/n}$. But then $M'=U'V'$ is an $m'\times n'$ submatrix of $M$ such that $\gamma_2(M')\leq ||U'||_{\row}||V'||_{\col}\leq \gamma/\varepsilon$. 
\end{proof}

\begin{proof}[Proof of \cref{cor:trace} assuming \cref{thm:mono_rectangle}]
By applying \cref{lemma:trace_to_gamma} with $\varepsilon=1/2$ to $M$, we get an $m/2\times n/2$ submatrix $M'$ such that $\gamma_2(M')\leq 2\gamma$. But then by \cref{thm:mono_rectangle}, $M'$ contains a $t\times u$ all-zeros or all-ones submatrix for some $\frac{t}{m},\frac{u}{m}\geq 2^{-O(\gamma^3)}$, finishing the proof.
\end{proof}

We now turn towards the proof of our main result,~\cref{thm:mono_rectangle}. To this end, we first argue that it is enough to consider the case $m=n$. Indeed, let $M$ be an $m\times n$ matrix such that $\gamma_2(M)\leq \gamma$. Let $M^{\otimes}$ be the matrix we get by repeating every row of $M$ exactly $n$ times and then repeating every column of the resulting matrix $m$ times, i.e.\ $M^{\otimes} = M \otimes J$ where $J$ is the $n \times m$ all-ones matrix. Then $M^{\otimes}$ is an $n'\times n'$ matrix with $n'=mn$ and $\gamma_2(M^{\otimes})=\gamma_2(M)\leq \gamma$. Assume that $M^{\otimes}$ contains a $t'\times t'$ all-zeros or all-ones submatrix, where $t'\geq cn'$ for some $c=c(\gamma)>0$. Then removing repeated rows and columns, this gives a $t\times u$ sized all-zeros or all-ones submatrix of $M$ with $t=\lceil\frac{t'}{n}\rceil$ and $u=\lceil\frac{t'}{m}\rceil$, so that $\frac{t}{m},\frac{u}{n}\geq c$.

Therefore, in the rest of the argument, we assume that $m=n$. Instead of \cref{thm:mono_rectangle}, we prove the slightly stronger result that if at least half of the entries are zero, then one can find a large all-zero submatrix.

\begin{theorem}\label{thm:all_zero}
     Let $M$ be an $n\times n$ Boolean matrix such that $\gamma_2(M)\leq \gamma$ and $p(M)\leq 1/2$. Then $M$ contains a $t\times t$ all-zeros submatrix for some $t=n2^{-O(\gamma^3)}$.
\end{theorem}

Note that \cref{thm:mono_rectangle} follows from this immediately: if $|M|\leq n^2/2$, we apply \cref{thm:all_zero} to find a large all-zero submatrix. However, if $|M|\geq n^2/2$, we apply \cref{thm:all_zero} to the matrix $J-M$. By noting that $\gamma_2(J-M)\leq \gamma+1$, this guarantees a large all-ones submatrix in $M$. The key technical result underpinning the proof of \cref{thm:all_zero} is the following lemma, which tells us that we can find a large submatrix of $M$ with significantly smaller $\gamma_2$-norm. Our main theorem then follows by repeated application of this lemma.

\begin{lemma}\label{lemma:main}
    Let $M$ be an $n\times n$ non-zero Boolean matrix such that $\gamma_2(M)\leq \gamma$, and $p(M)\leq 1/2$. Then $M$ contains an $m\times m$ submatrix $N$ such that $\gamma_2(N)\leq \gamma-\Omega(1/\gamma)$, $m\geq n2^{-O(\gamma)}$, and $p(N)\leq 1/2$.
\end{lemma}

\begin{proof}[Proof of \cref{thm:all_zero} assuming \cref{lemma:main}]
 Define a sequence of submatrices of $M$ as follows. Let $M_0=M$. If $M_i$ of size $n_i\times n_i$ is already defined such that $p(M_i)\leq 1/2$ and $M_i$ is non-zero, then let $M_{i+1}$ be the $n_{i+1}\times n_{i+1}$ submatrix of $M_i$ given by~\cref{lemma:main}, so that $\gamma_2(M_{i+1})<\gamma_2(M_{i+1})- c / \gamma_2(M_{i+1})$ for some constant $c > 0$, $p(M_{i+1})\leq 1/2$ and $n _{i+1}=n_i2^{-O(\gamma)}$. By induction, we have $\gamma(M_{i})\leq \gamma- c i /\gamma$ and $n_{i}=n2^{-O(\gamma i)}$. On the other hand, we trivially have $\gamma_2(M_i) \geq 0$ and so if $i > \gamma^2 / c$, then $M_i$ must not be defined. Since $M_{i+1}$ is only not defined in case $M_{i}$ is the all-zero matrix, if we let $i$ be the smallest integer such that $M_i$ is the all-zero matrix, then we have established that $i \leq \lceil \gamma^2 / c \rceil$ and so $M_i$ is the desired $n_i \times n_i$ matrix with $n_i = n 2^{-O(i \gamma)} = n2^{-O(\gamma^3)}$.
\end{proof}

The rest of this section is devoted to the proof of \cref{lemma:main}. The following definition and accompanying result are crucial.

\begin{definition}
    Let $M$ be an $m\times n$ Boolean matrix with factorization $M=UV$, and let $u_1,\dots,u_m$ be the row vectors of $U$, and $v_1,\dots,v_n$ be the column vectors of $V$. For $r\in [m]$, say that row $r$ is \emph{brilliant} (with respect to $(M,U,V)$) if $u_r\neq 0$, and denoting by $d_r$ the number of 1 entries in the $r$-th row of $M$, we have
    $$\sum_{i=1}^{m}\langle u_r,u_i\rangle^2\geq d_r.$$
    Similarly, for $c\in [n]$, say that column $c$ is  \emph{brilliant} (with respect to $(M,U,V)$) if $v_c\neq 0$, and denoting by $d'_c$ the number of 1 entries in the $c$-th column of $M$, we have
    $$\sum_{i=1}^{n}\langle v_c,v_i\rangle^2\geq d'_c.$$
\end{definition}

\begin{lemma}\label{lemma:brilliant}
Let $M$ be a non-zero Boolean matrix with factorization $M=UV$. Then there exists either a brilliant row or a brilliant column.
\end{lemma}

\begin{proof}
Assume that none of the rows or columns are brilliant. We make use of the following identity: if $x,y$ are real vectors of the same dimension, then $\langle xx^T,yy^T\rangle=\langle x,y\rangle^2$.

Let $u_1,\dots,u_m$ be the rows of $U$, and $v_1,\dots,v_n$ be the columns of $V$. We have
\begin{equation}\label{eq:brilliant}
    0\leq || \sum_{i=1}^m u_iu_i^T-\sum_{j=1}^n v_jv_j^T||_F^2=\sum_{r=1}^{m}\sum_{i=1}^m \langle u_r,u_i\rangle^2+\sum_{c=1}^{n}\sum_{j=1}^n \langle v_c,v_j\rangle^2-2\sum_{i=1}^{m}\sum_{j=1}^n\langle u_i,v_j\rangle^2.
\end{equation}
Here, $\sum_{i=1}^{m}\sum_{j=1}^n\langle u_i,v_j\rangle^2=|M|$ is the number of one entries of $M$. If none of the rows are brilliant, then $$\sum_{r=1}^{m}\sum_{i=1}^m \langle u_r,u_i\rangle^2<\sum_{r=1}^m d_r=|M|,$$
and similarly, if  none of the columns are brilliant, then 
$$\sum_{c=1}^{n}\sum_{j=1}^n \langle v_c,v_i\rangle^2<\sum_{c=1}^n d'_c=|M|.$$
This contradicts (\ref{eq:brilliant}), so we must have either a brilliant row or a brilliant column.
\end{proof}

\begin{proof}[Proof of \cref{lemma:main}]
Let $M_0$ be an $n_0\times n_0$ submatrix of $M$ such that $n_0=n2^{-O(\gamma)}$, and each row and each column of $M_0$ has at most $0.1n_0$ one entries. Applying \cref{lemma:sparsifying} with $\varepsilon=1/2$ and $\delta=0.1$, we can find such a submatrix. Write $M_0=U_0V_0$, where $||U_0||_{\row}^2\leq \gamma$ and $||V_0||_{\col}^2\leq \gamma$.

Repeat the following procedure. If $M_i=U_iV_i$ is already defined, where the size of $M_i$ is $m_i\times n_i$, proceed as follows. Stop if either (a) $M_i$ is an all-zero matrix, or (b) $m_i<n_0/2$, or (c) $n_i<n_0/2$. Otherwise, proceed depending on the following two cases.

\bigskip

\noindent
\textbf{Case 1.} There exists a brilliant row $r$ with respect to $(M_i,U_i,V_i)$.

Let $u_1,\dots,u_{m_i}$ be the rows of $U_i$. Let $u_j'$ be the projection of $u_j$ to the orthogonal complement of $u_r$, that is, 
$$u_j'=u_j-\frac{\langle u_r,u_j\rangle}{||u_r||^2}u_r.$$
Let $U_{i+1}$ be the matrix, whose rows are $u_1',\dots,u_{m_i}'$. Furthermore, let $V_{i+1}$ be the submatrix of $V_i$, in which we keep those columns $v$ of $V_i$  that satisfy $\langle u_r,v\rangle=0$. In other words, we keep the $j$-th column of $V_i$ if $(M_i)_{r,j}=0$. Finally, set $M_{i+1}=U_{i+1}V_{i+1}$. We observe that $M_{i+1}$ is a submatrix of $M_{i}$. This follows from the fact that if $\langle u_r,v\rangle=0$, then $\langle u_j',v\rangle=\langle u_j,v\rangle$ for all $j\in [m_i]$. Moreover, the size of $M_{i+1}$ is $m_i\times (n_{i}-t_i)$, where $t_i$ is the number of one entries in row $r$ of $M_i$. Furthermore, 
$$||u_j'||^2=||u_j-\frac{\langle u_r,u_j\rangle}{||u_r||^2}u_r||^2=||u_j||^2-\frac{\langle u_r,u_j\rangle^2}{||u_r||^2}\leq ||u_j||^2-\frac{\langle u_r,u_j\rangle^2}{\gamma},$$
from which
$$||U_{i+1}||_F^2=\sum_{j=1}^{m_i}||u_j'||^2\leq \sum_{j=1}^{m_i}||u_j||^2-\frac{1}{\gamma}\sum_{j=1}^{m_i}\langle u_r,u_j\rangle^2\leq ||U_{i}||_F^2-\frac{t_i}{\gamma},$$
where the last inequality follows by our assumption that row $r$ is brilliant.

\medskip

\noindent
\textbf{Case 2.} There exists no brilliant row  with respect to $(M_i,U_i,V_i)$. 

Then by \cref{lemma:brilliant}, there exists a brilliant column $c$. We proceed similarly as in the previous case, but with the roles of $U$ and $V$ swapped. Let $v_1,\dots,v_{n_i}$ be the columns of $V_i$, and let $v_j'$ be the projection of $v_j$ to the orthogonal complement of $v_c$. Let $V_{i+1}$ be the matrix, whose columns are $v_1',\dots,v_{n_i}'$. Furthermore, let $U_{i+1}$ be the submatrix of $U_i$, in which we keep those rows $u$ of $U_i$ such that $\langle u,v_c\rangle=0$. Finally, set $M_{i+1}=U_{i+1}V_{i+1}$.  We observe again that $M_{i+1}$ is a submatrix of $M_{i}$, and the size of $M_{i+1}$ is $(m_i-t_i)\times n_{i}$, where $t_i$ is the number of one entries in column $c$. Finally, $$||V_{i+1}||_F^2\leq ||V_{i}||_F^2-\frac{t_i}{\gamma}.$$

\medskip

We make some observations about our process.
\begin{enumerate}
    \item If at step $i$ we are in Case 1, then we get $M_{i+1}$ from $M_i$ by removing columns, and if we are in Case 2, we only remove rows.
    \item If for any $i$ we invoked Case 2, that is, there is no brilliant row at some step $i$, then there are no brilliant rows at any later step either. Indeed, in Case 2., we are only removing rows, so a non-brilliant row cannot become brilliant.
    \item The process ends in a finite number of steps. Indeed, at each step, we turn a non-zero vector of $U_i$ or $V_i$ into a zero vector by projection.
\end{enumerate}  

Let $i=I$ be the index for which we stopped, and recall that this means that either (a) $M_I$ is the all-zero matrix, or (b) $n_I<n_0/2$ or (c) $m_I<n_0/2$. Note that at each step $i$, the number of rows or columns decreases by some $t$ (possibly $t=0$), where $t$ is upper bounded by the number of one entries in a row or a column of $M_0$, in particular $t\leq 0.1n_0$. Hence, we must have that $m_I>n_0/4$ and $n_I>n_0/4$. If (a) $M_I$ is the all-zero matrix, we are done by taking $N$ to be any $n_0/4\times n_0/4$ submatrix of $M_I$, so assume that either (b) or (c) happened.

First, assume that (b) happened, so $n_I<n_0/2$. By (1) and (2), this is only possible if at every step, we invoked Case 1. Therefore, we have $m_I=n_0$, and $n_I=n_0-\sum_{i=0}^{I-1}t_i$, which implies $\sum_{i=0}^{I-1}t_i\geq n_0/2$. On the other hand, we have 
$$||U_{I}||_F^2\leq ||U_0||_F^2-\sum_{i=0}^{I-1}\frac{t_i}{\gamma}<||U_0||_F^2-\frac{n_0}{2\gamma}\leq n_0\gamma-\frac{n_0}{2\gamma}.$$
Let $U'$ be the submatrix of $U_I$ in which we keep those rows $u$ that satisfy $||u||^2\leq \gamma-\frac{1}{4\gamma}$, and let $m'$ be the number of rows of $U'$. Then 
$$n_0\left(\gamma-\frac{1}{2\gamma}\right)\geq ||U_I||_F^2\geq (n_0-m')\left(\gamma-\frac{1}{4\gamma}\right),$$
from which we conclude that $m'\geq n_0/(4\gamma^2)$. Set $N'=U'V_I$, then $N'$ is an $m'\times n_I$ submatrix of $M$ such that 
$$\gamma_2(N')\leq ||U'||_{\row}||V_I||_{\col}\leq \sqrt{\gamma-\frac{1}{4\gamma}}\cdot\sqrt{\gamma}=\gamma-\Omega(1/\gamma),$$ and $m',n_I=\Omega_{\gamma}(n)$. As $n_I\geq n_0/4$, each row of $N'$ has at most $0.1n_0\leq n_I/2$ one entries. Hence, $p(N')\leq 1/2$. Set $m=\min\{m',n_I\}\geq n_0/(4\gamma^2)=n2^{-O(\gamma)}$, then by a simple averaging argument, we can find an $m\times m$ submatrix $N$ of $N'$ such that $p(N)\leq 1/2$. This $N$ suffices.

Finally, we consider the case if (c) happened, that is, if $m_I<n_0/2$. In this case we proceed very similarly as in the previous case, so we only give a brief outline of the argument. Let $J$ be the first index for which  we invoked Case 2 in the process. Using (2), it follows that we invoked Case 2 for every index $i>J$ as well. We have $n_I=n_J\geq n_0/2$ and $m_I=n_0-\sum_{i=J}^{I-1}t_i$. As before, we conclude that
$$||V_{I}||_F^2\leq ||V_J||_F^2-\sum_{i=J}^{I-1}\frac{t_i}{\gamma}<||V_J||_F^2-\frac{n_0}{2\gamma}\leq n_J\gamma-\frac{n_J}{2\gamma}.$$
From this, we deduce that if $V'$ is the submatrix of $V_I$ formed by those columns $v$ that satisfy $||v||^2\leq \gamma-\frac{1}{4\gamma}$, then $V_J$ has at least $n_J/(4\gamma^2)\geq n_0/(8\gamma^2)$ columns. Setting $N'=U_IV'$, the rest of the proof proceeds in the same manner as in case (b).
\end{proof}

Finally, we prove \cref{cor:integer}.

\begin{proof}[Proof of \cref{cor:integer}]
Observe that if $M$ is an integer matrix such that $\gamma_2(M)\leq \gamma$, then all entries of $M$ are in the set $S=\{-\lfloor\gamma\rfloor,\dots,\lfloor\gamma\rfloor\}$. Hence, there exists some integer $t\in S$ such that at least $mn/(2\gamma+1)$ entries of $M$ are equal to $t$. Let $q\in \mathbb{R}[X]$ be the polynomial defined as $q(x)=1+c\prod_{k\in S\setminus\{t\}}(x-k)$, where $c\neq 0$ is chosen such that $q(t)=0$. Then $q(k)=1$ for every $k\in S\setminus\{t\}$, and $q(t)=0$. Let $N$ be the $m\times n$ Boolean matrix defined as $N_{i,j}=q(M_{i,j})$. Then writing $q(x)=\sum_{i=0}^{|S|-1}a_ix^i$, we have $N=\sum_{i=0}^{|S|-1}a_i M^{\circ i}$, where $M^{\circ i}$ denotes the Hadamard product $M\circ\dots\circ M$ with $i$ terms. Hence, by property (\ref{pr:hadamard}) of the $\gamma_2$-norm, we have 
$$\gamma':=\gamma_2(N)\leq \sum_{i=0}^{|S|-1}|a_i|\gamma^i=\gamma^{O(\gamma)}.$$
Here, $p(N)\leq 1-1/(2\gamma+1)$, so we can apply  \cref{lemma:sparsifying} to find a submatrix an $m'\times n'$ submatrix $N'$ such that $p(N')\leq 1/2$ and $\frac{m'}{m},\frac{n'}{n}\geq 2^{-O(\gamma^2 \gamma')}=2^{-\gamma^{O(\gamma)}}$. Finally, applying \cref{lemma:main}, we get that $N'$ contains an all-zero submatrix $N''$ of size $m''\times n''$ with $\frac{m''}{m'},\frac{n''}{n'}\geq 2^{-O(\gamma'^3)}=2^{-\gamma^{O(\gamma)}}$. But then $N''$ corresponds to a submatrix of $M$ of size $m2^{-\gamma^{O(\gamma)}}\times n2^{-\gamma^{O(\gamma)}}$ in which every entry is $t$, finishing the proof.
\end{proof}

\section{Construction of tight example}\label{sect:construction}

In this section, we present a construction of matrices with bounded $\gamma_2$-norm and no large all-ones and all-zeros submatrices.

\begin{proof}[Proof of \cref{prop:construction}]
 Let $\ell=\lfloor\gamma\rfloor\geq 3$, and let $k$ be an integer sufficiently large with respect to $\ell$. Let $S$ be a $k$ element ground set, let $p=k^{3/2-\ell}$, and let $\mathcal{F}$ be a random sample of the $\ell$-element subsets of $S$, where each $\ell$-element set is included independently with probability $p$. If we let $X=|\mathcal{F}|$, then $\mathbb{E}(X)=p\binom{k}{\ell}=\Omega_{\ell}(k^{3/2})$ and by standard concentration arguments, $\mathbb{P}(X>\mathbb{E}(X)/2)>0.9$. Let $Y$ be the number of pairs of sets in $\mathcal{F}$ whose intersection has size at least two. Then $\mathbb{E}(Y)<p^2k^{2\ell-2}=k$, so by Markov's inequality, $\mathbb{P}(Y<10k)\geq 0.9$. Furthermore, let $Y'$ be the number of pairs of sets in $\mathcal{F}$ whose intersection has size exactly 1. Then $\mathbb{E}(Y')\leq p^2 k^{2\ell-1}=k^2$, so by Markov's inequality, $\mathbb{P}(Y'<10k^2)\geq 0.9$.  Finally, let $T\subset S$ be any set of size $k/2$ and let $Z_T$ be the number of elements of $\mathcal{F}$ completely contained in $T$. Then $\mathbb{E}(Z_T)=p\binom{k/2}{\ell}<2^{-\ell}\mathbb{E}(X)$. By the multiplicative Chernoff inequality, we can write
$$\mathbb{P}(Z_T\geq 2\mathbb{E}(Z_T))\leq \exp\left(-\frac{1}{3}\mathbb{E}(Z_T)\right)\leq \exp(-\Omega_{\ell}(k^{3/2})).$$
Hence, as the number of $k/2$ element subsets of $S$ is at most $2^k$, a simple application of the union bound  implies $$\mathbb{P}(\forall T\subset S, |T|=k/2: Z_T\leq 2\mathbb{E}(Z_T))>0.9.$$
In conclusion, there exists a choice for $\mathcal{F}$ such that $X>\mathbb{E}(X)/2$, $Y<10k$, $Y'<10k^2$, and $Z_T\leq 2\mathbb{E}(Z_T)\leq 2\cdot 2^{-\ell}\mathbb{E}(X)$ for every $k/2$ element set $T$. For each  pair of sets intersecting in more than one element in $\mathcal{F}$, remove one of them from $\mathcal{F}$ and let $\mathcal{F}'$ be the resulting set. If we let $n=|\mathcal{F}'|/2=\Omega_{\ell}(k^{3/2})$, then $n>X/2-5k\geq \mathbb{E}(X)/4-5k =\Omega_{\ell}(k^{3/2})$, and thus $Z_T\leq 8\cdot 2^{-\ell}n$ for every $T$.

Define the $n\times n$ matrix $M$ as follows. Let $\mathcal{A}\cup\mathcal{B}$ be an arbitrary partition of $\mathcal{F}'$ into two $n$ element sets. Let $U$ be the $n\times k$ matrix whose rows are the characteristic vectors of the elements of $\mathcal{A}$, let $V$ be the $k\times n$ matrix whose columns are the characteristic vectors of the elements of $\mathcal{B}$, and set $M=UV$. As each row of $U$ and each column of $V$ is a zero-one vector with $\ell$ one entries, we have $\gamma_2(M)\leq ||U||_{\row}||V||_{\col}=\ell$. Also, $M$ is a Boolean matrix, which is guaranteed by the fact that any two distinct sets in $\mathcal{F}'$ intersect in 0 or 1 elements. The number of one entries of $M$ is at most $Y'<10k^2=o(n^2)$, which shows that $M$ contains no $n2^{-\gamma}\times n2^{-\gamma}$ sized all-ones submatrix if $n$ is sufficiently large. Finally, $M$ contains no $t\times t$ all-zeros submatrix if $t>8\cdot 2^{-\ell}n$. Indeed, a $t\times t$ all-zeros submatrix corresponds to subfamilies $\mathcal{A}'\subset \mathcal{A}$ and $\mathcal{B}'\subset \mathcal{B}$ of sizes $t$ such that every element of $\mathcal{A}'$ is disjoint from every element of $\mathcal{B}'$. In other words, if $T_1=\bigcup_{A\in\mathcal{A}'}A$  and $T_2=\bigcup_{B\in\mathcal{B}'} B$, then $T_1$ and $T_2$ are disjoint. But then at least one of $T_1$ or $T_2$ has size at most $k/2$, without loss of generality, $|T_1|\leq k/2$. The number of elements of $\mathcal{F}'$ contained in $T_1$ is at most $8\cdot 2^{-\ell}n$, so we indeed have $t\leq 8\cdot 2^{-\ell}n$.

\end{proof}

\section{Four cycle-free matrices}\label{sect:four_cycle}
In this section, we prove \cref{thm:four_cycle}. Most of this section is devoted to proving that four cycle-free matrices of average degree $d$ have $\gamma_2$-norm at least $\Omega(\sqrt{d})$. From this, \cref{thm:four_cycle} follows after a bit of work.

We adapt certain graph terminology for Boolean matrices. By the average degree of a matrix, we mean the average degree of the bipartite graph whose bi-adjacency matrix is $M$. Let $M$ be a matrix of average degree at least $d$. The first step is to find a submatrix of average degree $\Omega(d)$ where either every row contains $\Theta(d')$ one entries, or every column contains $\Theta(d')$ entries for some $d'=\Omega(d)$. Unfortunately, it is a well known result of graph theory \cite{JS} that it is not always possible to find a submatrix in which this is true for both the rows and columns simultaneously, which would also make our proof significantly simpler.

\begin{lemma}\label{lemma:regularize}
Let $M$ be a Boolean matrix of average degree at least $d$. Then $M$ contains a submatrix $N$ of average degree $d'\geq d/3$ such that every row and column of $N$ contains at least $d'/2$ one entries, and either $||N||_{\row}^2\leq 6d'$ or $||N||_{\col}^2\leq 6d'$.
\end{lemma}

\begin{proof}
Let $G$ be the bipartite graph, whose bi-adjacency matrix is $M$, and let $A$ and $B$ be the vertex classes of $G$. Our task is to show that $G$ contains an induced subgraph $G'$ of average degree  $d'\geq d/3$ such that $G'$ has minimum degree at least $d'/2$, and every degree in one of the parts is at most $6d'$. 

Let $G_0$ be an induced subgraph of $G$ of maximum average degree and let $d_0$ be the average degree of $G_0$, so that $d_0\geq d$. First, we note that $G_0$ has no vertex of degree less than $d_0/2$. Indeed, otherwise, if $v\in V(G_0)$ is such a vertex, then the average degree of $G_0-v$ (i.e., the graph we get by removing $v$) has average degree $2e(G_0-v)/(v(G_0)-1)>(2e(G_0)-d_0)/(v(G_0)-1)=d_0$.

Let $A_0\subset A,B_0\subset B$ be the vertex classes of $G_0$, and assume without loss of generality that $|A_0|\geq |B_0|$. Note that the number of edges of $G_0$ is $\frac{d_0}{2}(|A_0|+|B_0|)$. If we let $C\subset A_0$ be the set of vertices of degree more than $2d_0$, then $|C|\leq |A_0|/2$. Indeed, otherwise, the number of edges of $G_0$ is at least $2d_0|C|>d_0|A_0|\geq\frac{d_0}{2}(|A_0|+|B_0|)$, a contradiction. Let $A_1=A_0\setminus C$, and let $G_1$ be the subgraph of $G_0$ induced on $A_1\cup B_0$. The number of edges of $G_1$ is at least $d_0|A_1|/2\geq d_0 (|A_1|+|B_0|)/6$, so the average degree of $G_1$ is at least $d_0/3$. Let $G'$ be an induced subgraph of $G_1$ of maximum average degree, and let $d'$ be the average degree of $G'$. Then $d'\geq d_0/3$, and every vertex of $G'$ has degree at least $d'/2\geq d_0/6\geq d/6$. Furthermore, if $A'\subset A_1$ and $B'\subset B_0$ are the vertex classes of $G'$, then every degree in $A'$ is at most $2d_0\leq 6d'$. This finishes the proof.
\end{proof}

Now the idea of the proof is as follows. After passing to a submatrix $N$ which is close to regular from one side, say all columns have $\Theta(d')$ one entries, we use the fact that $$\gamma_2(N)\geq||N\circ (u v^T)||_{\tr}$$
for any choice of unit vectors $u$ and $v$ (of the appropriate dimensions), see property (\ref{pr:gamma}) of the $\gamma_2$-norm. We choose $u$ to be a vector whose entries are based on the degree distribution of the rows, and choose $v$ to be the normalized all-ones vector. Letting $A=N\circ (u v^T)$, we then inspect the matrices $B=AA^T$ and $B^2$. With the help of the Cauchy interlacing theorem, we show that the singular values of $A$ follow a certain distribution and thus find a lower bound for $||A||_{\tr}$.

\begin{lemma}\label{lemma:C4-free}
Let $M$ be a four cycle-free Boolean matrix of average degree at least $d$. Then $$\gamma_2(M)=\Omega(\sqrt{d}).$$
\end{lemma}
\begin{proof}
Let $N$ be a submatrix of $M$ satisfying the outcome of \cref{lemma:regularize}. Then if we let $d_0$ be the average degree of $N$, we have $d_0\geq d/3$, every row and column of $N$ contains at least $d_0/2$ one entries, and $||N||_{\row}^2\leq 6d_0$ or $||N||_{\col}^2\leq 6d_0$. Without loss of generality, we assume that $||N||_{\col}^2\leq 6d_0$. In what follows, we only work with the matrix $N$, and our goal is to prove that $\gamma_2(N)=\Omega(\sqrt{d_0})$, which will then imply that $\gamma_2(M)=\Omega(\sqrt{d})$. To simplify notation, we write $d$ instead of $d_0$.

Let the size of $N$ be $m\times n$, and recall that for any choice of $u\in \mathbb{R}^m$ and $v\in \mathbb{R}^n$ with $||u||=||v||=1$, we have
$$\gamma_2(N)\geq ||M\circ (uv^T)||_{\tr}.$$
Let $f$ be the number of one entries of $N$, and for $i=1,\dots,m$, let $d_i$ be the number of one entries of row $i$. Note that $f=d_1+\dots+d_m$. Let $u\in \mathbb{R}^m$ be defined as $u(i)=\sqrt{d_i/f}$ for $i\in [m]$, and let $v\in \mathbb{R}^n$ be defined as $v(i)=1/\sqrt{n}$. Then we have $||u||=||v||=1$. If we let $A=N\circ (uv^T)$, then $\gamma_2(N)\geq ||A||_{\tr}$, so it is enough to prove that $||A||_{\tr}=\Omega(\sqrt{d})$. Note that if we let $\sigma_1\geq \dots\geq \sigma_m\geq 0$ be the singular values of $A$ (possibly with zeros added to get exactly $m$ of them), then $\sigma_1^2,\dots,\sigma_m^2$ are the eigenvalues of $B=AA^T$. 

Define the auxiliary graph $H$ on $[m]$, where for $i,i'\in [m]$, $i\neq i'$, we put an edge between $i$ and $i'$ if there is some index $j\in [n]$ such that $N(i,j)=N(i',j)=1$. As $M$ is four cycle-free, there is at most one such index $j$ for every pair $(i,i')$. With this notation, we can write
$$B(i,i')=\frac{1}{fn}\begin{cases} d_i^2 &\mbox{ if }i=i'\\
                        \sqrt{d_id_{i'}} &\mbox{ if }i\sim i'\\
                        0 &\mbox{ otherwise.}\end{cases}$$
For $t=1,\dots,\lceil\log_3 n\rceil=:p$, let $I_t\subset [m]$ be the set of indices $i$ such that $3^{t-1}\leq d_i\leq 3^t$. Then $I_1,\dots,I_p$ forms a partition of $[m]$, and we note that $I_t$ is empty if $t\leq \log_3 d-1$. Let $B_t=B[I_t\times I_t]$, then $B_t$ is a principal submatrix of $B$.

\begin{claim}
At least $\Omega(|I_t|)$ eigenvalues of $B_t$ are at least $\Omega(3^{2t}/(fn))$. 
\end{claim}

\begin{proof}
Let $s=|I_t|$, $D=3^{t-1}$, and let $\lambda_1\geq \dots\geq \lambda_s\geq 0$ be the eigenvalues of $B_t$. Then 
\begin{equation}\label{equ:trace}
    \lambda_1+\dots+\lambda_s=\tr(B_t)=\frac{1}{fn}\sum_{i\in I_t}d_i^2\geq \frac{sD^2}{fn},
\end{equation}
and
$$\lambda_1^2+\dots+\lambda_s^2=||B_t||_F^2.$$
Here,
$$||B_t||_F^2=\sum_{i,i'\in I_t}B(i,i')^2=\frac{1}{(fn)^2}\left[\sum_{i\in I_t}d_i^4+\sum_{i\sim i',i,i'\in I_t}d_id_{i'} \right]\leq \frac{1}{(fn)^2}\left[81sD^{4}+18e(H[I_t])D^{2}\right],$$
where $e(H[I_t])$ denotes the number of edges of the subgraph of $H$ induced on the vertex set $I_t$. Let $G$ be the bipartite graph, whose bi-adjacency matrix is $N[I_t\times [n]]$. Then $e(H[I_t])$ is the number of pairs $\{i,i'\}\in I_t^{(2)}$ such that $i$ and $i'$ has a common neighbour in $G$. As every column of $N$ has at most $6d$ one entries, every vertex in $[n]$ has degree at most $6d$ in $G$. Thus, for each $i\in I_t$, there are at most $6dd_i\leq 18dD$ vertices $i'\in I_t$ which have a common neighbour with $i$. Therefore, $e(H[I_t])\leq 12dDs\leq 36D^2s$, where we used in the last inequality that $s=0$ unless $t\geq \log_3 d-1$. In conclusion, we proved that 
$$\lambda_1^2+\dots+\lambda_s^2\leq \frac{1000sD^4}{(fn)^2}.$$
Let $C=2000$, and let $r\leq s$ be the largest index such that $\lambda_r\geq \frac{CD^2}{fn}$. By the previous inequality, we have $r\leq \frac{1000s}{C^2}$. But then by the inequality between the arithmetic and square mean, we have $$\lambda_1+\dots+\lambda_r\leq r^{1/2}(\lambda_1^2+\dots+\lambda_r^2)^{1/2}\leq r^{1/2}\left(\frac{1000sD^4}{(fn)^2}\right)^{1/2}\leq \frac{sD^2}{2fn}.$$
Hence, comparing this with (\ref{equ:trace}), we deduce that
$$\lambda_{r+1}+\dots+\lambda_{s}\geq \frac{sD^2}{2fn}.$$
As $\frac{CD^2}{fn}\geq \lambda_{r+1}\geq\dots\geq \lambda_s$, this is only possible if at least $\frac{s}{4C}$ among $\lambda_{r+1},\dots,\lambda_s$ is at least $\frac{D^2}{4fn}$. This finishes the proof.
\end{proof}
As $B_t$ is a principal submatrix of $B$, its eigenvalues interlace the eigenvalues of $B$, see \cref{lemma:cauchy}. Therefore, the previous claim implies that at least $c|I_t|$ eigenvalues of $B$ are at least $c3^{2t}/(fn)$ for some absolute constant $c>0$, and thus at least $c|I_t|$ singular values of $A$ are at least $c3^t/\sqrt{fn}$. 

We are almost done. Note that 
$$f=\sum_{i=1}^md_i\leq \sum_{t=1}^p 3^{t+1}|I_t|.$$
In order to bound $||A||_{\tr}=\sigma_1+\dots+\sigma_m$, we observe that for every $t$, if $|I_t|\geq \max\{|I_{t+1}|,\dots,|I_p|\}=:z_t$, then $$\sum_{i=z_t+1}^{|I_t|}\sigma_i\geq \frac{c3^t}{\sqrt{fn}}\cdot (c|I_t|-c|I_{t+1}|-\dots-c|I_p|).$$ Hence, 
\begin{align*}
||A||_{\tr}&=\sum_{i=1}^m \sigma_i\geq \sum_{t=1}^p \frac{c3^t}{\sqrt{fn}}\cdot (c|I_t|-c|I_{t+1}|-c|I_2|-\dots-c|I_{p}|)\\
&\geq \frac{c^2}{\sqrt{fn}}\sum_{t=1}^p |I_t|(3^t-3^{t-1}-\dots-3-1) \geq \frac{c^2}{2\sqrt{fn}}\sum_{t=1}^p3^t |I_t|\geq \frac{c^2\sqrt{f}}{6\sqrt{n}}.
\end{align*}
Here, in the first inequality, we use that if $|I_t|< z_t$, then the contribution of the $t$-th term in the sum is anyway negative. Finally, recall that every column of $N$ contains at least $d/2$ one entries, so $f\geq dn/2$. Therefore, $||A||_{\tr}\geq \frac{c^2}{12}\sqrt{d}$, finishing the proof.
\end{proof}

\begin{proof}[Proof of \cref{thm:four_cycle}]
First, we prove that  $\gamma_2(M)\leq 2\sqrt{d}$. Let $G$ be the bipartite graph, whose bi-adjacency matrix is $M$, and let $A$ and $B$ be the vertex classes of $G$, where $A$ corresponds to the rows of $M$, and $B$ corresponds to the columns. Then there is an ordering $<$ of $A\cup B$ such that every vertex $v\in A\cup B$ has at most $d$ neighbours larger than $v$ in this ordering. Let $G_1$ be the subgraph of $G$ on vertex set $A\cup B$ in which we keep those edges $ab$ with $a\in A$ and $b\in B$, where $a<b$, and let $G_2$ be the graph containing the rest of the edges. Let $M_i$ be the bi-adjacency matrix of $G_i$ for $i=1,2$, then $M=M_1+M_2$, every row of $M_1$ has at most $d$ one entries, and every column of $G_2$ has at most $d$ one entries. In particular, $\gamma(M)\leq \gamma(M_1)+\gamma(M_2)\leq 2\sqrt{d}$.

Second, we prove that $\gamma_2(M)=\Omega(\sqrt{d})$. If we let $N$ be a submatrix of $M$ of minimum degree $d$, then the average degree of $N$ is at least $d$. Hence, \cref{lemma:C4-free} implies the desired lower bound by noting that $\gamma_2(M)\geq \gamma_2(N)$. 
\end{proof}

Finally, we prove \cref{thm:zarankiewicz}. The key graph theoretical result tying \cref{thm:four_cycle} and \cref{thm:all_zero} together is the following lemma of \cite{HMST0}.

\begin{lemma}[Hunter, Milojevi\'c, Sudakov, Tomon \cite{HMST0}]\label{lemma:dense_C4}
For every $k$ there exists $c=c(k)>0$ such that the following holds. Let $G$ be a bipartite graph of average degree $d$, which contains no four cycle-free induced subgraph with average degree at least $k$. Then $G$ contains a subgraph on at most $d$ vertices with at least $cd^2$ edges.
\end{lemma}

\begin{proof}[Proof of \cref{thm:zarankiewicz}]
Let $G$ be the bipartite graph, whose bi-adjacency matrix is $M$, and let $d\geq Ct$ be the average degree of $G$, where $C=C(\gamma)$ is specified later. By \cref{thm:four_cycle}, there exists $c_0>0$ such that $G$ contains no four-cycle free induced subgraph with average degree at least $k=c_0\gamma^2$. If we let $c=c(k)$ be the constant guaranteed by \cref{lemma:dense_C4}, then $G$ contains a subgraph $H$ with at most $d$ vertices, and at least $cd^2$ edges. Then $H$ corresponds to a $d_1\times d_2$ submatrix $N$ of $M$ with at least $cd^2$ one entries, where $d_1+d_2\leq d$. We want to find a large all-ones submatrix of $N$. We may assume that $d_1=d_2=d$ by adding all-zero rows and columns to $N$. If we let $N'=J-N$, then $\gamma_2(N')\leq \gamma+1$, and $N'$ has at least $cd^2$ zero entries. In other words, $p(N')\leq 1-c$. By applying \cref{lemma:sparsifying} to $N'$ with $\varepsilon=c$ and $\delta=1/2$, we can find an $\alpha d\times \alpha d$ sized submatrix $N''$ of $N$ such that $\alpha=2^{-O(\gamma/c^2)}$ and $p(N'')<1/2$. If we let $n'=\alpha d$, then \cref{thm:all_zero} implies that $N''$ contains a $t'\times t'$ all-zeros submatrix for some $t'=n'2^{-O(\gamma^3)}\geq \lambda d$, where $\lambda=\lambda(\gamma)>0$ only depends on $\gamma$. In conclusion, choosing $C=1/\lambda$, we have $t'\geq t$, and thus  $M$ contains a $t\times t$ all-ones submatrix.
\end{proof}

\section{Graphs of bounded smallest eigenvalue}\label{sect:smallest_eig}

In this section, we prove  \cref{thm:smallest_eigenval}. First, we present a technical lemma, which we use to construct a large class of graphs that are forbidden in graphs of bounded smallest eigenvalue.

\begin{lemma}\label{lemma:forbidden_subgraph}
Let $H$ be a graph satisfying the following. The vertex set of $H$ can be partitioned into three $k$ element sets $A,B,C$ such that there are no edges between $A$ and $B$, but the bipartite subgraph between $A\cup B$ and $C$ is complete (and we make no assumption about the edges in $A,B$ or $C$). Then the smallest eigenvalue of $H$ is  at most $-k/3$.
\end{lemma}

\begin{proof}
Let $A$ be the adjacency matrix of $H$, and let $x\in\mathbb{R}^{V(H)}$ be the vector defined as $x(v)=1/\sqrt{3k}$ if $v\in A\cup B$, and $x(v)=-1/\sqrt{3k}$ if $v\in C$. Observe that $||x||=1$. Let $-\lambda$ be the smallest eigenvalue of $A$, then
\begin{align*}
    -\lambda&\leq x^TAx= \frac{2}{3k}(e(G[A])+e(G[B])+e(G[C])+e(G[A,B])-e(G[A\cup B,C]))=\\
    &=\frac{2}{3k}(e(G[A])+e(G[B])+e(G[C])-2k^2)\\
    &\leq \frac{2}{3k}\left(\frac{3k^2}{2}-2k^2\right)=-\frac{k}{3}.
\end{align*}

\end{proof}

Next, we show that the adjacency matrix of a graph with bounded smallest eigenvalue has bounded $\gamma_2$-norm as well.

\begin{lemma}\label{lemma:smallest_eigenvalue}
Let $\lambda>0$, let $G$ be a graph and let $A$ be the adjacency matrix of $G$. If the smallest eigenvalue of $G$ is at least $-\lambda$, then $\gamma_2(A)\leq 2\lambda$.
\end{lemma}

\begin{proof}
If $M$ is an $n \times n$ symmetric positive semidefinite matrix, it is not hard to show that $\gamma_2(M) = \max_{1\leq i\leq n}{M_{i,i}}$, which was first proved by Schur \cite{Sch} in 1911. Therefore, as $A+\lambda I$ is positive semidefinite, we have $\gamma_2(A + \lambda I) = \lambda$. By the subadditivity of the norm, $$\gamma_2(A) \leq \gamma_2(A + \lambda I) + \gamma_2(-\lambda I) = 2 \lambda.$$
\end{proof}

\begin{proof}[Proof of  \cref{thm:smallest_eigenval}]
We may assume that $d$ is sufficiently large with respect to $\lambda$, otherwise the statement is trivial. Let $A$ be the adjacency matrix of $G$, then $\gamma_2(A)\leq 2\lambda$. Hence, by  \cref{thm:zarankiewicz}, $A$ contains a $t\times t$ all-ones submatrix for some $t=\Omega_{\lambda}(d)$. But then this corresponds to a complete bipartite subgraph in $G$ with vertex classes $X,Y$ such that $|X|=|Y|=t$. We assume that $d$ is sufficiently large such that $t\geq 3\lambda+1$ is satisfied. 

\begin{claim}
   If $k=\lceil3\lambda+1\rceil$, then $X$ cannot contain two disjoint $k$ element sets with no edges between them.
\end{claim}

\begin{proof}
    Assume to the contrary that there exist $A,B\subset X$ with no edges between them, and $|A|=|B|=k$. If $C$ is any $k$ element subset of $Y$, then $A\cup B\cup C$ spans a subgraph $H$ of $G$ fitting the  description of \cref{lemma:forbidden_subgraph}. Thus the smallest eigenvalue of $H$ is at most $-k/3< -\lambda$. But then as $H$ is an induced subgraph of $G$, the Cauchy interlacing theorem (\cref{lemma:cauchy}) implies that the smallest eigenvalue of $G$ is less than $-\lambda$, contradiction.
\end{proof}
Let $G'$ the complement of $G[X]$, and let $D$ be the average degree of $G'$. The adjacency matrix of $G'$ has $\gamma_2$-norm at most $2\lambda+2$. Hence, applying \cref{thm:zarankiewicz} to $G'$, we deduce that $G'$ contains a complete bipartite graph with vertex classes of size $u=\Omega_{\lambda}(D)$. But this gives two $u$-element sets in $G[X]$ with no edges between them. Therefore, by the previous claim, $u\leq 3\lambda+1$, and in particular $D=O_{\lambda}(1)$. Applying Tur\'an's theorem to $G[X]$, we conclude that $G[X]$ contains a complete subgraph of size at least
$$\frac{|X|}{D+1}=\Omega_{\lambda}(d),$$
finishing the proof.
\end{proof}

\section{Inverse theorem for MaxCut}\label{sect:maxcut}

In this section, we prove  \cref{thm:main_MaxCut}. Let us introduce some notation. Given a graph $G$, the \emph{MaxCut} of $G$ is the maximum number of edges in a \emph{cut} of $G$, where a cut is a partition of $V(G)$ into two sets, with all edges having exactly one endpoint in both parts. We denote by $\mc(G)$ the MaxCut of $G$, and furthermore, if $G$ has $m$ edges, we define the \emph{surplus} of $G$ as 
$$\surp(G)=\mc(G)-m/2.$$
In case the vertices of $G$ are partitioned randomly, the expected number of edges in the resulting cut is $m/2$. This implies that $\surp(G)\geq 0$. The following is equivalent to  \cref{thm:main_MaxCut}. 

\begin{theorem}\label{thm:main_maxcut}
Let $G$ be a  graph with $m$ edges such that $\surp(G)\leq \alpha \sqrt{m}$. Then $G$ contains a clique of size at least $2^{-O(\alpha^9)}\sqrt{m}$.
\end{theorem}

We prepare the proof of this theorem by collecting a few simple properties of the surplus.

\begin{lemma}\label{lemma:surplus_sum}
Let $G$ be a graph, and let $V_1,\dots,V_k$ be disjoint subsets of $V(G)$. Then
$$\surp(G)\geq \sum_{i=1}^k\surp(G[V_i]).$$
\end{lemma}

\begin{proof}
For $i=1,\dots,k$, let $(A_i,B_i)$ be a partition of $V_i$ witnessing the MaxCut of $G[V_i]$. Define the partition $(X,Y)$ of $V(G)$ as follows: for every $i\in [k]$, let either $A_i\subset X$ and $B_i\subset Y$ or $A_i\subset Y$ and $B_i\subset X$ independently with probability $1/2$. Also, add every vertex not contained in $\bigcup_{i=1}^k V_i$ to either $X$ or $Y$ with probability $1/2$. Observe that if $f$ is an edge of $G$ not contained in $G[V_i]$ for any $i\in [k]$, then $f$ is cut with probability $1/2$. Otherwise, if $f\subset V_i$, then $f$ is cut if and only if $f$ is cut by $(A_i,B_i)$. Hence, the expected size of the cut $(X,Y)$ is exactly $|E(G)|/2+\sum_{i=1}^k \surp(G[V_i]))$.
\end{proof}

\begin{lemma}[Erd\H{o}s, Gy\'arf\'as, Kohayakawa \cite{EGyK97}]\label{lemma:surp_n}
Let $G$ be a graph on $n$ vertices with no isolated vertices. Then $\surp(G)\geq n/6$.
\end{lemma}

The main technical lemma underpinning the proof of \cref{thm:main_maxcut} is the following result showing that the surplus can be lower bounded by the energy of $G$. The \emph{energy} of $G$ is the trace-norm of the adjacency matrix of $G$, or equivalently, the sum of absolute values of the adjacency matrix, and it is denoted by $\mathcal{E}(G)$.  A similar, but weaker result was proved recently by R\"aty and Tomon \cite{RatyTomon}, who showed that if $G$ has $n$ vertices, then $\surp(G)=\Omega(\mathcal{E}(G)/\log n)$. Here, we prove that the $1/\log n$ factor can be removed. Our proof follows the semidefinite programming approach of \cite{RatyTomon}. However, to remove the $1/\log n$ factor, which comes from an application of the Graph Grothendieck inequality \cite{AMMN}, we instead combine this idea with the MaxCut approximation approach of Goemans and Williamson \cite{GoemansWilliamson}. 

\begin{lemma}\label{lemma:surplus_energy}
Let $G$ be a graph. Then $\surp(G)=\Omega(\mathcal{E}(G))$.
\end{lemma}

\begin{proof}
We recall the following identity: if $x,y\in \mathbb{R}^n$, then $\langle xx^T,yy^T\rangle=\langle x,y\rangle^2$.

We may assume that $G$ contains no isolated vertices, as the removal of such vertices does not change the energy or the surplus. Let $\{1,\dots,n\}$ be the vertex set of $G$, and let $A$ be its adjacency matrix. Let $\lambda_1\geq \dots\geq \lambda_n$ be the eigenvalues of $A$ with a corresponding orthonormal basis of eigenvectors $v_1,\dots,v_n$. Define the symmetric positive semidefinite matrix
$$M=\sum_{i: \lambda_i<0} v_iv_i^T.$$
We make a number of observations about $M$. First of all, 
$$\langle A,M\rangle=\left\langle \sum_{i=1}^n \lambda_i v_iv_i^T,\sum_{i:\lambda_i<0}v_iv_i^T\right\rangle=\sum_{i=1}^n\sum_{j:\lambda_j<0}\lambda_i\langle v_i,v_j\rangle^2=\sum_{i:\lambda_i<0}\lambda_i=-\mathcal{E}(G)/2.$$
Moreover, 
$$||M||_F^2=\langle M,M\rangle=\sum_{i:\lambda_i<0}1\leq n.$$
Lastly, $M_{i,i}\leq 1$ for $i\in [n]$, as $I-M=\sum_{i:\lambda_i\geq 0}v_iv_i^T$ is positive semidefinite. Let $M'$ be the $n\times n$ matrix such that $M'_{i,i}=1$ for $i\in [n]$ and $M'_{i,j}=M_{i,j}$ if $i,j\in [n],i\neq j$. Then $M'$ is positive semidefinite, $\langle A,M'\rangle=\langle A,M\rangle$, and $||M'||_F^2\leq 2n$. In particular, there exist unit vectors $x_1,\dots,x_n\in\mathbb{R}^n$ such that $M'_{i,j}=\langle x_i,x_j\rangle$ for every $i,j\in [n]$.

Using the vectors $x_1,\dots,x_n$, we describe a probability distribution on the cuts of $G$ as follows. Let $\mathbf{H}$ be a random linear hyperplane in $\mathbb{R}^n$, then $\mathbf{H}$ cuts the $n$-dimensional space into two parts. We define the partition $(X,Y)$ of $V(G)$  such that $X$ contains the vertices corresponding to the vectors $x_1,\dots,x_n$ in one of the parts, while $Y$ is the set of vertices corresponding to the vectors in the other part. Formally, we take $\mathbf{u}$ randomly on the unit sphere $\mathbb{S}^{n-1}$ from the uniform distribution, and set $$X=\{i\in [n]: \langle x_i,\mathbf{u}\rangle\geq 0\}\mbox{\ \ \ and\ \ \ }Y=\{i\in [n]:\langle x_i,\mathbf{u}\rangle<0\}.$$ Given an edge $f=\{a,b\}$, consider the probability that $f$ is cut by $(X,Y)$. This probability can be calculated easily as follows: let $\alpha=\arccos(\langle x_a,x_b\rangle)$ be the angle between $x_a$ and $x_b$, then
$$\mathbb{P}(f\mbox{ is cut by }(X,Y))=\frac{\alpha}{\pi}.$$
By the Taylor expansion of the function $\arccos(.)$, we can write $$\arccos(t)=\frac{\pi}{2}-t-O(t^2).$$ The error term can be improved to $O(t^3)$, however, quadratic decay already suffices for our purposes. From this, we get
$$\mathbb{P}(f\mbox{ is cut by }(X,Y))=\frac{1}{2}-\frac{\langle x_a,x_b\rangle}{\pi}-O(\langle x_a,x_b\rangle^2).$$
Let $\mathbf{T}$ be the number of edges in the cut $(X,Y)$, then
\begin{align*}
\mathbb{E}(\mathbf{T})&=\sum_{\{a,b\}\in E(G)} \mathbb{P}(\{a,b\}\mbox{ is cut by }(X,Y))=\sum_{\{a,b\}\in E(G)}\frac{1}{2}-\frac{\langle x_a,x_b\rangle}{\pi}-O\left(\langle x_a,x_b\rangle^2\right)\\
&=\frac{e(G)}{2}-\frac{\langle A,M'\rangle}{\pi}-O(||M'||_F^2)=\frac{e(G)}{2}+\frac{\mathcal{E}(G)}{2\pi}-O(n).
\end{align*}
Therefore, as there exists a cut $(X,Y)$ with at least $\mathbb{E}(\mathbf{T})$ edges, we arrive to the inequality $$\surp(G)\geq \frac{\mathcal{E}(G)}{2\pi}-O(n).$$
By  \cref{lemma:surp_n}, we also have $\surp(G)=\Omega(n)$, which combined with the previous inequality gives the desired bound $\surp(G)=\Omega(\mathcal{E}(G))$.
\end{proof}

\begin{proof}[Proof of  \cref{thm:main_maxcut}]
Let $n$ be the number of vertices of $G$. Without loss of generality, we may assume that $G$ does not contain isolated vertices. Then, by  \cref{lemma:surp_n}, we have $\surp(G)\geq n/6$. Hence, $$\sqrt{m}\leq n\leq 6\alpha\sqrt{m},$$
where the first inequality is true for every graph, and the second inequality is due to our assumption $\surp(G)\leq \alpha\sqrt{m}$. By \cref{lemma:surplus_energy}, there exists a constant $c_0>0$ such that $c_0\surp(G)\geq \mathcal{E}(G)$, so $$\mathcal{E}(G)\leq c_0\alpha\sqrt{m}\leq c_0\alpha n.$$ Let $A$ be the adjacency matrix of $G$, then $\mathcal{E}(G)=||A||_{\tr}\leq c_0\alpha n$. Hence, by  \cref{lemma:trace_to_gamma}, for every $\varepsilon>0$, there exists a $n'\times n'$ submatrix $B$ of $A$ such that $\gamma_2(B)\leq c_0\alpha/\varepsilon$ and $n'\geq (1-\varepsilon)n$. Here, as the removal of $\varepsilon n$ rows and columns can destroy at most $2 \varepsilon n^2$ one entries, we have
$$|B|\geq |A|-2\varepsilon n^2\geq |A|-72\alpha^2\varepsilon m=m(2-72\alpha^2\varepsilon).$$
Choosing $\varepsilon=1/72\alpha^2$, we get a submatrix $B$ such that 
$$|B|\geq m\geq \frac{1}{36\alpha^2}(n')^2$$ and $\gamma:=\gamma_2(B)\leq 72c_0\alpha^3$. Next, we want to find a large submatrix $B'$ such that $p(B')\geq 1/2$. In order to do this, we apply  \cref{lemma:sparsifying}  to $J-B$. As $p(J-B)\leq 1-\frac{1}{36\alpha^2}$, we can find such a submatrix $B'$ of size $n''\times n''$, where $n''=2^{-O(\gamma\alpha^4)}n'=2^{-O(\alpha^7)}\sqrt{m}$.  Now we can apply \cref{thm:all_zero} to $J-B'$ to conclude that $B'$, and thus $A$, contains $t\times t$ sized all-ones submatrix with $t=2^{-O(\gamma^3)}n''=2^{-O(\alpha^9)}\sqrt{m}$.

This all-ones submatrix corresponds to a complete bipartite subgraph of $G$, let $X$ and $Y$ be its vertex classes, and note that $|X|=|Y|=t$. Let $H=G[X\cup Y]$. By \cref{lemma:surplus_sum}, $\surp(G)\geq \surp(H)$, so $\surp(H)\leq \alpha \sqrt{m}$. On the other hand, by considering the cut $(X,Y)$, we have $\mc(H)\geq t^2$, so
$$\surp(H)=\mc(H)-\frac{|E(H)|}{2}\geq t^2-\frac{t^2+e(G[X])+e(G[Y])}{2}=  \frac{t^2}{2}-\frac{e(G[X])}{2}-\frac{e(G[Y])}{2}.$$
Write $e(G[X])=\binom{t}{2}-x$ and $e(G[Y])=\binom{t}{2}-y$. Then using the inequality $\surp(H)\leq \alpha\sqrt{m}$, the previous inequality becomes
$$\alpha\sqrt{m}\geq \frac{t+x+y}{2}\geq \frac{x}{2}.$$
In particular, we showed that the complement of $G[X]$ has at most $2\alpha\sqrt{m}$ edges. By Tur\'an's theorem (\cref{thm:turan}), then $G[X]$ contains a clique of size at least
$$\frac{t^2}{4\alpha\sqrt{m}+t}=2^{-O(\alpha^9)}\sqrt{m}.$$
This finishes the proof.
\end{proof}

\section*{Acknowledgments} We would like to thank Hamed Hatami and Marcel Goh for sharing their upcoming manuscript with us, and for valuable discussions. LH is thankful to Nathan Harms for his feedback on an early draft of the paper, which helped improve the exposition.

\bibliographystyle{alpha}
\bibliography{maxnorm}

\end{document}